\documentclass[11pt]{amsart}
\usepackage[toc,page]{appendix}
\usepackage{amsmath,amssymb,latexsym}
\usepackage[small]{caption}
\usepackage{graphicx,color,mathrsfs,tikz}
\usepackage{subfigure,color}
\usepackage{cite}
\usepackage{multicol}
\usepackage[colorlinks=true,urlcolor=blue,
citecolor=red,linkcolor=blue,linktocpage,pdfpagelabels,
bookmarksnumbered,bookmarksopen]{hyperref}
\usepackage[italian,english]{babel}
\usepackage{enumitem}
\usepackage[left=2.65cm,right=2.65cm,top=2.9cm,bottom=2.9cm]{geometry}
\usepackage[hyperpageref]{backref}

\usepackage[colorinlistoftodos,prependcaption]{todonotes}

\numberwithin{equation}{section}
\newtheorem{theorem}{Theorem}[section]
\newtheorem{proposition}[theorem]{Proposition}

\newtheorem{definition}[theorem]{Definition}
\theoremstyle{definition}
\newtheorem{example}[theorem]{Example}

\newtheorem{problem}{Problem}

\newtheorem{remark}[theorem]{Remark}
\newcommand{\eps}{\varepsilon}
\newcommand{\A}{{\mathcal A}}
\newcommand{\B}{{\mathcal B}}

\renewcommand{\H}{{\mathcal H}}
\newcommand{\I}{{\mathcal I}}
\renewcommand{\L}{{\mathcal L}}
\newcommand{\Co}{{\mathcal C}}

\newcommand{\R}{{\mathbb R}}

\renewcommand{\div}{{\rm div}}

\newcommand{\loc}{{\rm loc}}

\renewcommand{\theta}{\vartheta}

\newcommand{\Rn}{{\mathbb R^{n}}}
\newcommand{\eqlab}[1]{\begin{equation}  \begin{aligned}#1 \end{aligned}\end{equation}} 
\newcommand{\bgs}[1]{\begin{equation*} \begin{aligned}#1\end{aligned}\end{equation*}} 
  \newcommand{\sys}[2][]{\begin{equation*}#1  \left\{\begin{aligned}#2\end{aligned}\right.\end{equation*}}

\def\XXint#1#2#3{{\setbox0=\hbox{$#1{#2#3}{\int}$ }
\vcenter{\hbox{$#2#3$ }}\kern-.6\wd0}}

\title[The stickiness phenomena]{The stickiness phenomena of nonlocal minimal surfaces: new results and a comparison with the classical case}

\author[C. Bucur]{Claudia Bucur}
\address[C. Bucur]{Isituto Nazionale di Alta Matematica ``Francesco Severi''}
\email{claudia.bucur@aol.com}

\thanks{The author is a member
	of {\em Gruppo Nazionale per l'Analisi Ma\-te\-ma\-ti\-ca, la Probabilit\`a e le loro Applicazioni} (GNAMPA) 
	of the {\em Istituto Nazionale di Alta Matematica} (INdAM)}

\thanks{I am grateful to Prof. Fausto Ferrari for the invitation to give a Pini Analysis Seminar and to the University of Bologna for the hospitality. I also deeply thank Luca Lombardini, Enrico Valdinoci and Serena Dipierro for their precious suggestions and careful reading of this note.}


\begin{document}

\begin{abstract}
We discuss in this note the stickiness phenomena for nonlocal minimal surfaces. Classical minimal surfaces in convex domains do not stick to the boundary of the domain, hence examples of stickiness can be obtained only by removing the assumption of convexity. On the other hand, in the nonlocal framework, stickiness is ``generic''.  We provide various examples from the literature, and focus  on the case of complete stickiness in highly nonlocal regimes.\\

\bigskip

In questa nota ci occupiamo del fenomeno di attaccamento al bordo delle superfici minime nonlocali. Generalmente, le superfici minime classiche non presentano tale fenomeno in un dominio convesso, pertanto alcuni esempi di attaccamento al bordo si ottengono solamente in assenza della condizione di convessit\`{a}. Per contro, nel contesto nonlocale, l'attaccamento al bordo \`{e} un comportamento ``generico''. Proporremo  diversi esempi dalla letteratura, per di pi\`{u} incentrati sul caso di attaccamento completo al bordo,  nei cosiddetti regimi altamente nonlocali.    
\end{abstract}

\maketitle
\tableofcontents
 
The problem regarding surfaces with least area among those enclosed by a given curve is one of the first questions that arose in the calculus of variations. Named after Plateau due to his experiments on soap films and bubbles, carried out by the French physicist  in the nineteenth century, the question on minimal surfaces  actually dates back to Lagrange (1760).  Plateau's problem received some first answers in the thirties in $\R^3$, by Douglas and Rad\`{o}.  In its full generality, it was attacked by several outstanding mathematicians, who tackled the problem from different, very ingenious prospectives, such as, to mention the most famous:  Almgren and Allard, introducing the theory of varifolds, Federer and Fleming developing the theory of currents, Reifenberg applying methods from algebraic topology, De Giorgi working with the perimeter operator (see the beautiful Introduction of \cite{MinGiusti} for more details). The achievements and the history on Plateau's and closely related problems are inscribed in many branches of mathematics, such as geometric measure theory (actually born to study this problem), differential geometry, calculus of variations, potential theory, complex analysis and mathematical physics.   The story is far from being over, since the various fields of study are now days very active, they present a variety of new accomplishments and still pose many open problems. The reader can consult the following books, surveys and papers \cite{ColdMin,Perez} for classical minimal surfaces, \cite{NevMar,Nevmarr} for the Willmore conjecture and min-max theory approach, \cite{DeLellis1,DeLellis2,ambtt} for recent achievements in geometric measure theory, and can find further references of their interest therein.

This note will just ``scratch the surface'' in the attempt to give an introduction to the argument. We will focus on the case of co-dimension one, following the approach of the Italian mathematician Ennio De Giorgi, who defines minimal surfaces as boundaries of sets which minimize a perimeter operator inside a domain, among sets with given boundary data. In this context, the main argument on which we focus is the so-called stickiness phenomenon: in some occasions, minimal surfaces are forced by the minimization problem and the boundary constraints to ``attach'' to the boundary of the given domain. 

For classical minimal surfaces, this phenomena is rare and happens only in ``extreme'' conditions. In convex domains, minimal surfaces reach transversally the boundary of the domain, so stickiness is not contemplated. Furthermore, minimal graphs (i.e., minimal surfaces which are also graphs) always attain in convex domains  their (continuous) boundary data in a continuous way. We will present in Example \ref{msstick1} a situation in which stickiness may happen if the domain is not convex.

On the other hand, nonlocal minimal surfaces, introduced as the nonlocal (fractional) counterpart of the classical ones, typically stick. Even taking the ``best'' domain (i.e. a ball) and a very nice exterior data, surprisingly the stickiness phenomenon is not only possible, but it appears in many circumstances.  In this note, we gather several examples from the literature and we discuss in more detail the case of complete stickiness (that is, when the nonlocal minimal surface attaches completely to the boundary of the domain), in highly nonlocal regimes (that is, for small values of the fractional parameter).

\bigskip

In the rest of the paper, we set the following notations:
\begin{itemize}
\item points in $\Rn$ as $x=(x_1,\dots,x_n)$ and points in $\R^{n+1}$ as $X=(x,x_{n+1})$,
\item the $(n-1)$-Hausdorff measure as $\H^{n-1}$,
\item the complementary of a set $\Omega \subset \Rn$ by $\Co \Omega=\Rn \setminus \Omega$,
\item the ball of radius $r>0$ and center $x\in \Rn$ as
\[ B_r(x)=\big\{ y\in \Rn \; \big| \; |y-x|<r\big\}, \qquad B_r:= B_r(0),\]
\item the area of the unit sphere as
\[ \omega_n:=\H^{n-1}(\partial B_1).\]

\end{itemize}

\section{An introduction to classical minimal surfaces}

 Just to give a basic idea, the approach of De Giorgi to minimal surfaces can be summarized as follows.

Consider  an  open set $\Omega \subset \Rn$ and a measurable set $E \subset \Rn$. If the set $E$ has $C^2$ boundary inside $\Omega$, the area of the boundary of $E$ in $\Omega$ is given by
\eqlab{\label{smo} \mbox{Area}(\partial E \cup \Omega) = \H^{n-1}(\partial E\cap \Omega).}
On the other hand, in case $E$ does not have a smooth boundary, one can introduce a weak version of the perimeter.

\begin{definition} Let $\Omega\subset \Rn$ be an open set and $E\subset\Rn $ be a measurable set. The perimeter of $E$ in $\Omega$ is given by
\eqlab{ \label{bv} P(E,\Omega):= \sup_{g \in C_c^1(\Omega,\Rn), |g|\leq 1} \int_{E} \div \, g \, dx.}
\end{definition}

Notice that when $E$ has $C^2$ boundary, the expected \eqref{smo} is recovered.  Indeed, taking any $g \in C_c^1(\Omega, \Rn)$, we have that 
\bgs{ \int_{E} \mbox{div} g \, dx = \int_{\partial E} g \cdot \nu_E \, d \H^{n-1}  ,}
using the divergence theorem and denoting $\nu_E$ as the exterior normal derivative to $E$. Then
\bgs{ P(E,\Omega) = &\;  \sup_{g \in C_c^1(\Omega,\Rn), |g|\leq 1}\int_{E} \mbox{div} g \, dx \\
		=&\;  \sup_{g \in C_c^1(\Omega,\Rn),|g|\leq 1}  \int_{\partial E} g \cdot \nu_E \, d \H^{n-1}  
		\\
		 \leq  &\;  \int_{\partial E\cap \Omega}  d \H^{n-1}    =\H^{n-1}(\partial E \cap \Omega) .}
		 A particular choice of $g$ leads to the opposite inequality and proves the statement. Since $E$ has smooth boundary, $\nu_E$ is a $C^1$ vector valued function, so it can be extended to a vector field $N \in C^1(\Rn, \Rn)$, with $\|N\|\leq 1$. Consider a cut-off function $\eta \in C^\infty_c(\Omega)$ with $|\eta|\leq 1$ and use $g =\eta N$. Then  
		 \bgs{  P(E,\Omega)=&\; \sup_{g \in C_c^1(\Omega,\Rn),|g|\leq 1}  \int_{\partial E} g \cdot \nu_E \, d \H^{n-1}  
		 \\ 
		 \geq &\;  \sup_{\eta \in C^\infty_c(\Omega), |\eta|\leq 1} \int_{\partial E} \eta \, d \H^{n-1} 
				 \\
		 =&\;\H^{n-1}(\partial E \cap \Omega).}
		We recall that 
		 the space of functions of bounded variation $BV(\Omega)$ is defined 
		 as
		 \[ BV(\Omega):=\big\{ u\in L^1(\Omega) \; \big| \; [u]_{BV(\Omega)}<\infty\big\}, 
\]
where
\[ [u]_{BV(\Omega)}=  \sup_{g \in C_c^1(\Omega,\Rn), |g|\leq 1}  \int_{\Rn} u \, \div g\, dx,\]
and that $BV(\Omega)$ is a Banach space with the norm
\[ \|u\|_{BV(\Omega)} =\|u\|_{L^1(\Omega)} + [u]_{BV(\Omega)}.\]
It is evident then that the perimeter of a set $E\subset \Rn$ is  the total variation of its characteristic function, i.e. the BV norm 
of the characteristic function of $E$
\sys[\chi_E(x)=]{ &1, && x\in E
					\\ 
					&0, && x \in \Co E,}
					so we can write that
\eqlab{\label{bvv} P(E,\Omega)= [\chi_E]_{BV(\Omega)}.} 


\medskip

Sets of (locally) finite perimeter, or of (local) finite total variation (i.e., sets with $P(E,\Omega)<\infty$) bear the name of the Italian mathematician Renato Caccioppoli, who introduced them in 1927. 
Among sets of finite perimeter, minimal sets are the ones that minimize the perimeter with respect to some fixed ``boundary'' data. Of course, we work in the class of equivalence of sets, that is, we identify sets which coincide up to sets of measure zero. Maintaining the same perimeter, in principle  sets could have completely different topological boundaries. 
That is why in this note we assume measure theoretic notions (see for instance \cite[Chapter 3]{MinGiusti}, \cite[Section 1.2]{bucluk}).  
In order to avoid any technical difficulties, a set is defined as minimal in $\Omega$ if it minimizes the perimeter among competitors with whom it coincides outside of $\Omega$. Precisely:

\begin{definition} \label{min} 
Let $\Omega \subset \Rn$ be an open, bounded set, $\mathcal B$ be an open ball such that $\bar \Omega \subset \mathcal B$ and $E\subset \Rn$ be a measurable set. Given $E_0:= E \cap( \B \setminus \Omega)$, then $E$ is a minimal set in $\Omega$ with respect to $E_0$ if $P(E,\B)<\infty$ and
\[P(E, \B)\leq P(F,\B) \]
 for any $F$ such that
\[ F\cap (\B \setminus  \Omega) =E_0. \]
\end{definition}

Since the perimeter is a local operator, the ``boundary'' data considered is in the proximity of $\partial \Omega$. 
That is why it is not necessary to require that $E=F$ in the whole complementary of $\Omega$, and it suffices to consider the ball $\B$ (and hence, not to worry about what happens far away from $\Omega$).  We make the choice of a ball $\B$ for simplicity, one could consider an open set $\mathcal O \supset \bar \Omega$, or for some $\rho>0$  the set $\Omega_\rho:=\{ x\in \Rn \; | \; d(x,\partial \Omega)=\rho\}.$

\medskip

In the space $BV(\Omega)$, it is also quite natural to prove the existence of minimal sets. The lower semi-continuity of $BV(\Omega)$ functions and the fact that sequences of sets with uniformly bounded perimeters are precompact in the $L^1_{loc}$ topology, allow to employ  direct methods in the calculus of variations (see, for instance, \cite[Theorem 1.20]{MinGiusti}, \cite[Theorem 3.1]{TeoAlessio}) and to prove the existence of a minimal set, for a given $E_0$ of finite perimeter.
\medskip

\begin{theorem} Let $\Omega \subset \Rn$ be a bounded open set and let $E_0\subset \Co \Omega$ be a set of finite perimeter. Then there exists $E$ a minimal set in $\Omega$ with respect to $E_0$. 
\end{theorem}
  The arduous part is to prove {regularity}: are the boundaries of these sets actually smooth (almost everywhere)? This is indeed the case, and this entitles the theory to refer to boundaries of minimal sets as {minimal surfaces}. 
The boundary regularity of minimal sets can be summed up in the following theorem.

\begin{theorem}\label{minreg}
Let $\Omega\subset \Rn$ be a bounded open set and $E$ be a minimal set. Then $\partial E$ is smooth, up to a closed, singular set of Hausdorff dimension at most $n-8$.
\end{theorem}

In other words, minimal surfaces are smooth for $n\leq 7$ (and they are actually analytical). In $\R^8$, there exist minimal surfaces with singular points. A well known example is Simons cone, which is a minimal cone (with a singularity in the origin): 
\[ \mathcal S=\big\{x=(x,y)\in \R^4\times \R^4  \; \big| \; |x|=|y| \big\}.\] 
   
%
%
   \subsection{Minimal graphs}
   
   In the first part of this Section, we have introduced the perimeter operator and have discussed some essential properties of the following problem.
   
\begin{problem} \label{pb1} Given $\Omega \subset \Rn$ a bounded open set, $\B$ an open ball such that $\bar \Omega \subset \B$ and $E_0 \subset \B \setminus \Omega$ a set of finite perimeter, find 
 \[ \min \big\{ P(E,\B) \; \big| \; P(E,\B)<\infty,  E=E_0 \mbox{ in } \B \setminus \Omega\big\} .\] 
\end{problem} 

  A special case of minimal sets that we look for are minimal subgraphs, case in which the minimal surfaces are called minimal graphs.  We recall the space of Lipschitz continuous functions, denoted by $C^{0,1}(\Omega)$, defined for some open set $\Omega \subset\Rn$ by continuous functions with finite Lipschitz constant
\[ [u]_{C^{0,1}(\Omega)} =\sup_{x,y\in \Omega, x\neq y}\frac{|u(x)-u(y)|}{|x-y|}.\]

  The problem of looking for minimal graphs  in $C^{0,1}(\Omega)$ can be stated as follows.\\

\begin{problem} \label{pb2} Given $\Omega \subset \Rn$ a bounded open set with Lipschitz continuous boundary, and fixing $\varphi$ smooth enough on $\partial \Omega$, find $u\in C^{0,1}(\Omega)$ that realizes 
 \[ \min_{u=\varphi  \mbox{ on } \partial \Omega} \A(u,\Omega),\]
 where $\A$ is the area operator, defined as
 \eqlab{ \label{are} \A(u,\Omega)= \int_{\Omega} \sqrt{1+|Du|^2}\, dx.}
\end{problem}
 
 Notice that the area operator is well defined for $u\in C^{0,1}(\Omega)$.
%

 Existence and uniqueness (given that the area functional is convex) can be proved in the following context (see \cite[Theorem 12.10]{MinGiusti}).
 \begin{theorem}
 Let $\Omega$ be a bounded open set with $C^2$ boundary of non-negative mean curvature, and $\varphi \in C^2(\Rn)$. Then Problem \ref{pb2} is uniquely solvable in $C^{0,1}(\Omega)$. 
 \end{theorem}
 
  Tools of regularity of nonlinear partial differential equations in divergence form allow then to go from Lipschitz to analyticity in the interior and, in the hypothesis of the above theorem, to $C^2(\bar \Omega)$, settling the question on regularity of minimizers of  Problem \ref{pb2} (see \cite[Theorem 12.11, 12.12]{MinGiusti}). 
 
 \begin{theorem}
 Let $u \in C^{0,1}(\Omega)$ a solution of Problem \ref{pb2}. Then $u$ is analytic in $\Omega$. If moreover, $\partial \Omega$ and $\varphi$ are of class $C^{k,\alpha}$, with $k\geq 2$, then $u\in C^{k,\alpha}(\bar \Omega)$.  
 \end{theorem}
 
We stress out that in order to ensure existence of a solution of Problem \ref{pb2}, the condition that the mean curvature of $\partial \Omega$ is nowhere negative is necessary. We  provide here \cite[Example 12.15]{MinGiusti} (see also \cite[Example 1.1]{GiustiDirect}, \cite[Section 2.3]{giaq }) showing that for a domain whose boundary is somewhere non-positive, the solution may not exist, or may not be regular up to the boundary. The following example is depicted in Figure \ref{stt}.

\begin{example}\label{msstick}
Let $0<\rho<R$, $M>0$ be fixed, and let $A^R_\rho$ be the annulus
\[ A_R^\rho =\big\{ x\in \R^2  \; \big| \; \rho<|x|<R\big\}.\]
Define $\varphi $ on the boundary of $A^{\rho}_R$ as
\sys[\varphi(x)=]{& 0, && \mbox{ for } |x|=R
				\\	
				&M, && \mbox{ for } |x|=\rho.}
If $u$ is a minimum for the area in $A_R^{\rho}$, then the spherical average of $u$
\[ v(r) := \frac{1}{2\pi} \int_0^{2\pi} u(r,\theta) \, d\theta \]
decreases the area. Indeed, given the strict convexity of the area functional by, Jensen's inequality one gets that
 \[ \A(v,A_R^\rho)<\A(u,A_R^\rho).\] 
 This implies that the minimum must be radial, i.e. $u=u(r)$.
The area functional can then be written as
\[ F(u)= 2\pi \int_{\rho}^R r\sqrt{1+(u'(r))^2}\, dr,\] 
with Euler-Lagrange equation implying that $ru'/\sqrt{1+u'^2}$ is a constant, hence
\[ \frac{r u'(r)}{ \sqrt{1+(u'(r))^2}} = -c,\]
with $c\in[0, \rho] $ (positive since $u$ is non-increasing in $r$)  to be determined using the boundary conditions. 
The ODE, combined with $u(R)=0$, has the unique solution
\[ u(r)= c \log \frac{\sqrt{R^2-c^2} +R }{\sqrt{r^2-c^2} +r}.\]
One notices that the map
\[ f(c):=c \log \frac{\sqrt{R^2-c^2} +R }{\sqrt{\rho^2-c^2} +\rho}\] is non-decreasing in $[0,\rho]$, thus 
\[ \sup_{0\leq c\leq \rho} u(\rho)  = \sup_{c\in [0,\rho)} f(c)= \rho \log \frac{\sqrt{R^2-\rho^2} + R }{\rho}:=M_0,\]
with $M_0=M_0(R,\rho)$. 
However, by boundary conditions, one should have $u(\rho)=M$, thus a solution exists if only if $M_0\geq M$. Furthermore, notice that
\begin{itemize}
\item if $M_0 < M$, Problem \ref{pb2} does not have a solution;
\item if $M_0=M$,  thus when
\eqlab{\label{12} u(r)= \rho \log \frac{\sqrt{R^2-\rho^2} + R }{\sqrt{r^2-\rho^2} +r } }
we have that
\[ \lim_{r\searrow \rho} |u'(r)| =\infty,\]
implying that $u$ is not smooth up to the boundary.
\end{itemize}
  
\end{example}

     \medskip
          
   Taking into account  Example \ref{msstick}, we see that looking for a  minimum in $C^{0,1}(\Omega)$ can lead to a problem without any classical solution.  Another formulation can be considered for Problem \ref{pb2}, which  for the existence does not require non-negative mean curvature of $\Omega$ and relaxes the condition on the boundary data.  As with general sets, one works in the space of functions of bounded variation.
     For $u \in BV(\Omega)$, the area functional is defined as
     \eqlab{ \label{fff} \A(u,\Omega) = \sup_{g \in C_c^{\infty}(\Omega, \R^{n+1}) , |g|\leq 1} \int_{\Omega} g_{n+1} + u\, \div g \, dx,
     }
     with $\div g= \sum_{i=1}^n \partial_i g_i(x)$.
   Notice that for $u\in C^{0,1}(\Omega)$, Definition \ref{are} is recovered. 
     
With definition \eqref{fff}, the problem can be considered in this way (see \cite[14.4]{MinGiusti}).
     
       \begin{problem}\label{pb3}
   Let $\Omega\subset \Rn$ be a bounded open set, $\mathcal B $ be an open ball containing $\bar \Omega$ and let $\varphi \in W^{1,1}(\B \setminus \Omega)$. Find
   \[ \min	\big\{ \A(u,\B) \; \big| \; u\in BV(\mathcal B), u=\varphi \mbox{ in } \B\setminus \bar \Omega \big\} .\]
   \end{problem}
   
  Problem \ref{pb3} can be reformulated. Notice that
   \eqlab{ \label{af} \A(u,\B) = &\; \A(u,\Omega) + \A(u,\B \setminus \bar \Omega) + \int_{\partial \Omega} |u-\varphi| \, d \H^{n-1} 
   		\\
   		=&\; \A (u,\Omega) + \A( \varphi, \B \setminus \bar \Omega) + \int_{\partial \Omega} |u-\varphi| \, d\H^{n-1}.
   }
   Since $\varphi$ is fixed outside of $\Omega$, minimizing $u $ in $\B$ with exterior data $\varphi$ boils down to minimizing both the area  of $u$ in $\Omega$ and the area along the vertical wall $\partial \Omega \times \R$, lying between the graph of $\varphi$ and  $u$.
       The existence for any smooth set $\Omega$ is settled in the next Theorem, see \cite[Theorem 14.5]{MinGiusti}.

   \begin{theorem} For $\Omega$ with Lipschitz continuous boundary, there exists a solution of Problem \ref{pb3}.
   \end{theorem}

\begin{remark}  \label{nmg} Notice the resemblance of Problem \ref{pb3} with Problem \ref{pb1}. The similitude does not stop at the way the problem is defined: for sets that are graphs, the two formulations are actually equivalent. This follows after some considerations:
   \begin{enumerate}
   \item  defining the subgraph of $u\in BV(\Omega)$ as
   \[ Sg (u,\Omega)= \big\{ (x,x_{n+1}) \in \Omega\times \R \subset \R^{n+1}\; \big| \; x_{n+1}<u(x)\big\},\]
   it holds that
   \[ \A(u,\Omega)=P(Sg (u,\Omega), \Omega \times \R),\]
   \item given a set $F$ in a cylinder, then the perimeter decreases by replacing $F$ by a suitable subgraph, obtained with a ``vertical rearrangement'' of the set $F$(check \cite[Lemma 5.1]{TeoAlessio}, \cite[Lemma 14.7, Theorem 14.8]{MinGiusti}).  
\item  observe that the domain in which we minimize the perimeter in the class of subgraphs is unbounded, so additional care is needed to deal with local minimizers (we say that $u$ is a local minimizer in $\Omega$ if it minimizes the functional in any set compactly contained in $\Omega$). 
        \end{enumerate}
   \end{remark}      
In particular, finding a minimal graph is equivalent to finding a local minimizer of the perimeter in the class of subgraphs (\cite[Theorem 14.9]{MinGiusti}). 
Precisely:
\begin{theorem}
Let $u\in BV_{\loc}(\Omega)$ be a local minimum for the area functional. Then $Sg(u,\Omega)$ minimizes locally the perimeter in $\Omega\times \R$.
\end{theorem}

Since for graphs Problem \ref{pb1} and Problem \ref{pb3} are equivalent, regularity of general minimal surfaces applies to minimal graphs. Actually, purely functional techniques are used to prove that minimal graphs are smooth in any dimension \cite[Theorem 14.13]{MinGiusti}.

\begin{theorem}
Let $u\in BV_{\loc}(\Omega)$ locally minimize the area functional. Then $u$ is analytical inside $\Omega$.
\end{theorem}
 
 On the other hand, looking at boundary regularity, \cite[Theorem 15.9]{MinGiusti} states that:
 
 \begin{theorem}\label{regmg}
 Let $\Omega\in \Rn$ be a bounded open set with $C^2$ boundary, and let $u$ solve Problem \ref{pb3}. Suppose that $\partial \Omega$ has non-negative mean curvature near $x_0$ and that $\varphi$ is continuous at $x_0$. Then
 \[ \lim_{x\to x_0} u(x)= \varphi(x_0) .\]
 \end{theorem}
 
 The above theorem can actually be stated for domains $\Omega$ with Lipschitz boundary, by using a suitable notion of mean curvature. Also, notice that asking for non-negative mean curvature is more general than asking $\Omega$ to be convex. 
 
  A more attentive look at Theorem \ref{regmg} allows us to conclude that 
 in general, for continuous boundary data $\varphi$ and for  convex domains, the stickiness phenomena does not happen for minimal graphs. We will see that the situation dramatically changes for nonlocal minimal graphs. 
 
 On the other hand, looking at Example \ref{msstick}, one can provide an example of stickiness in non-convex domains.

\begin{example}\label{msstick1}
Let $0<\rho<R$, $M>0$ be fixed, and let $A^R_\rho$ be the annulus
\[ A_R^\rho =\big\{ x\in \R^2  \; \big| \; \rho<|x|<R\big\}.\] 
Define $\varphi $ as
\sys[\varphi(x)=]{& 0, && \mbox{ for } x\in \Co B_R
				\\	
				&M, && \mbox{ for } x\in \bar B_\rho,}
				and let $u(x)$ be the minimum of the area functional, defined by \eqref{12} as
\[ u(x)= \rho \log \frac{\sqrt{R^2-\rho^2} + R }{\sqrt{|x|^2-\rho^2} +|x| } .\]
Consider 
\sys[v(x):=]{&u(x) , && \rho \leq |x|\leq R\\
		& \varphi(x), && x \in B_{\rho}\cup \Co \bar B_R.}
	Notice that according to \eqref{af} we have that
	\bgs{ \A(v,B_{R+2})=&\; \A(v,A_R^\rho) + \A(v,  B_{R+2} \setminus \bar {A_R^\rho} ) + \int_{\partial A_R^\rho} |v-\varphi| \, d\H^{n-1}
	\\
	=&\; \A(u,A_R^\rho) 	+ \A(\varphi,  B_{R+2} \setminus \bar {A_R^\rho} ) + (M_0-M) \omega_n \rho^{n-1}.
	}
	Now, $u$ is a minimum  for the area in  $A_R^\rho$ (as shown in Example \ref{msstick}), the contribution of $\varphi$ is fixed, and $M_0$ is the highest possible value that $u$ can reach. This implies that $v$ is a solution of Problem \ref{pb3}. In this case, we notice that on $\partial B_\rho \times \R$ the solution $v$ sticks at the boundary, that $v$ is not continuous across the boundary, and the subgraph of $v$ has a vertical wall along the boundary of the cylinder in which we minimize. See Figure \ref{stt}.
		\end{example}
	
\begin{figure}[htpb] 
	\begin{multicols}{2}
    \includegraphics[width=\linewidth]{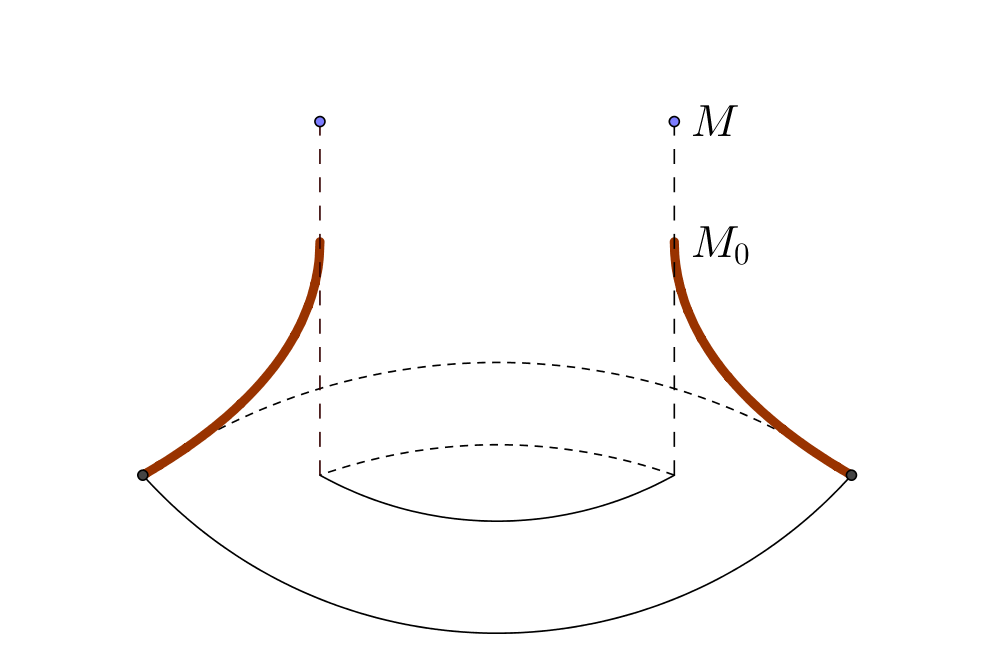}\par 	
    \includegraphics[width=\linewidth]{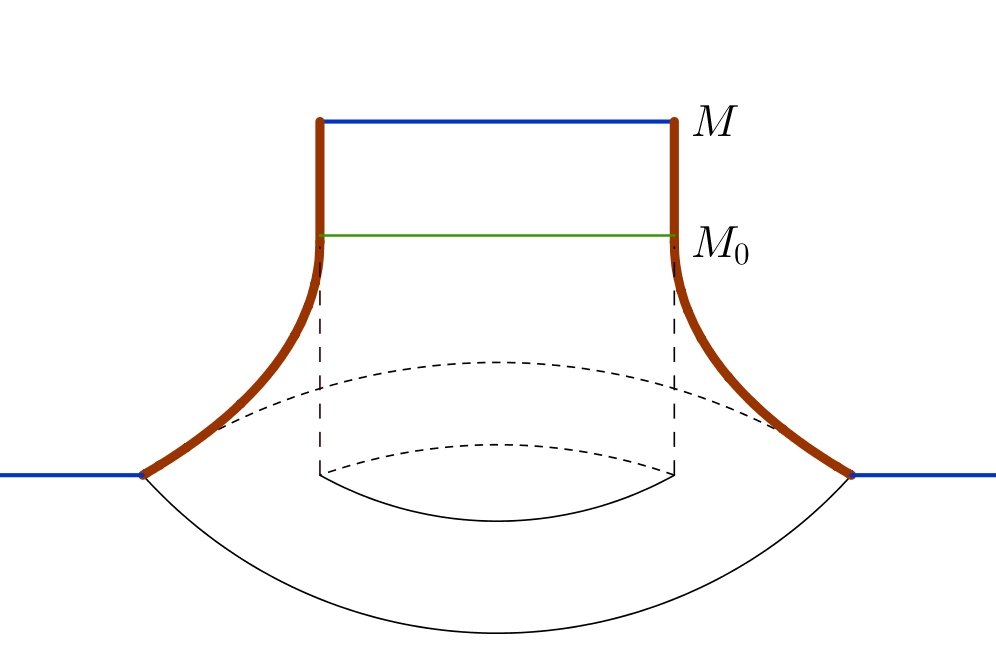}\par 	
    \end{multicols}
 \caption{The geometric construction in Examples \ref{msstick} and \ref{msstick1}}   
\label{stt}
	
\end{figure} 

\section{An introduction to nonlocal minimal surfaces}
 Justified by nonlocal phase transition problems and by imaging processing, one is led to introduce a nonlocal (and fractional version) of the perimeter. This was admirably accomplished in the seminal paper \cite{nms} by Caffarelli, Roquejoffre and Savin in 2010. The readers can check also the beautiful  and useful  review \cite{senonlocal}.\\
 Roughly speaking, one would like to have a definition of the nonlocal perimeter that takes into account long-range interactions between points in the set and in its complement, in the whole space, weighted by the their mutual distance. The goal is then to minimize such a perimeter in a domain $\Omega \subset \Rn$ among all competitors coinciding outside of $\Omega$, in a similar way to Definition \ref{min}. Notice now that in the nonlocal framework the data coming from far away plays a role, so the ``boundary'' data $E_0$ is given in the whole of $\Rn\setminus \Omega$ and the data even very distant from $\Omega$ gives a contribution. 
  
    To arrive at the definition of fractional perimeter introduced in \cite{nms}, one could start from \eqref{bvv}  and make use of a ``fractional counterpart'' of the $BV$ semi-norm. Notice that $ W^{1,1}(\Omega) \subset BV(\Omega)$, hence a good candidate turns out to be the {Gagliardo} $W^{s,1}$ semi-norm.
   For some given $s\in (0,1)$, we recall that for a measurable function $u\colon \Rn \to \R$ 
    \[ [u]_{W^{s,1}(\Omega) }=\int_{\Omega}\int_{\Omega} 
    \frac{ |u(x)-u(y)|}{|x-y|^{n+s}} \, dx \,dy.\]     
    Informally thus (because these quantities may well be infinite), the fractional perimeter is given by the $W^{s,1}$ semi-norm of the characteristic function of the set $E$
    \[ P_s (E,\Omega)= \frac12 \left([\chi_E]_{W^{s,1}(\Rn)} - [\chi_E]_{W^{s,1} (\Co \Omega)}\right).\]
    Of course, it would not be enough to take the $W^{s,1}$ semi-norm only in $\Omega$, because all far away information would be lost. Nonetheless, one excludes the interactions $\Co \Omega$ with $\Co \Omega$. This is due to the fact that in the minimization problem the data outside of the domain $\Omega$ is fixed, and so is that contribution. All in all, the fractional perimeter is defined as follows.
    
    \begin{definition}\label{pssss}
    Let $s\in (0,1)$ be fixed, $\Omega\subset \Rn$ be an open set and $E \subset \Rn$ be a measurable set. Then
    \[ P_s(E,\Omega)= \frac12\iint_{\R^{2n} \setminus (\Co \Omega)^2} \frac{|\chi_{E}(x) -\chi_E(y)|}{|x-y|^{n+s}} \, dx \, dy.\]
    \end{definition}
      In the above definition, notice that only the interactions between $E$ and its complement survive. Thus, denoting for two disjoint sets  $A, B \subset \Rn$
    \[ \L_s(A,B) =\int_A \int_B\frac{ dx \, dy }{|x-y|^{n+s}}  \] 
     we can write
       \eqlab{\label{pss} P_s(E,\Omega)= P_s^{L} (E,\Omega) + P_s^{NL} (E,\Omega),} 
     where we separate the  ``local'' and the ``nonlocal'' contributions to the perimeter (see Figure \ref{sp})
    \[  P_s^{L} (E,\Omega) := \L_s(E\cap \Omega,\Co E \cap \Omega),\]
         \[ P_s^{NL} (E, \Omega) := \L_s(\Co E\cap \Omega,E \cap \Co \Omega)  +\L_s (E\cap \Omega, \Co E \cap \Co \Omega).\] 
    \begin{figure}[htpb]
	\centering
	\includegraphics[width=0.8\textwidth]{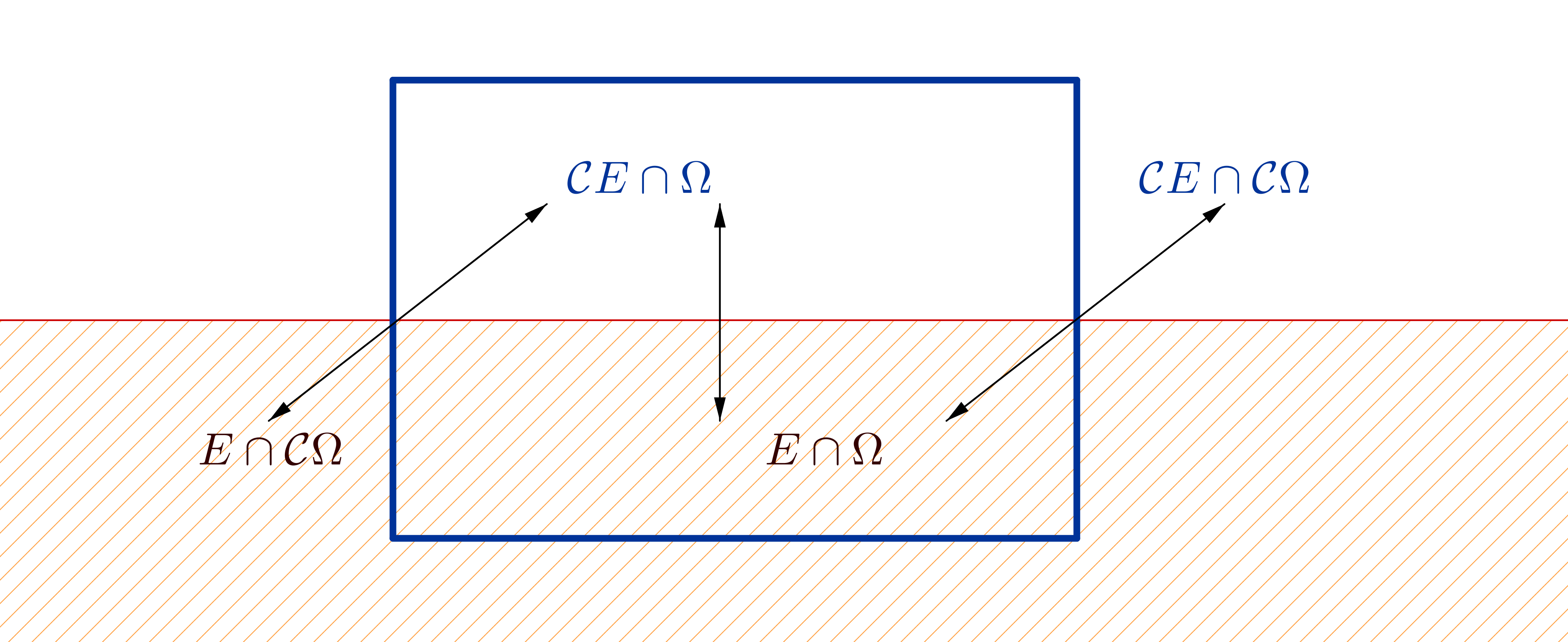}
	\caption{The contributions to the fractional perimeter}   
	\label{sp}
\end{figure}

     As a remark, it holds that
    \[ W^{1,1}(\Omega) \subset BV(\Omega)\subset \bigcap_{s\in (0,1)} W^{s,1}(\Omega), \]
    in particular if $E$ has finite perimeter, then it has finite fractional perimeter, for every $s\in(0,1)$ (on the other hand, the converse is not true).
    One notices that sending that $s\searrow 1$, the local perimeter comes up. This further justifies the fractional perimeter as a good generalization, in this sense, of the classical perimeter. As a matter of fact,  in \cite{uniform} the authors prove, under  local regularity assumptions on $\partial E$, that for  $s \nearrow 1$, the limit of
  $(1-s)P_s(E,B_1)$  goes to the classical $P(E,B_1)$ (the result in  the $\Gamma$-convergence sense is reached in \cite{DePhil}). The optimal result (in the pointwise sense) can be found in \cite[Theorem 1.6]{fractalLuk} (which is based on the previous\cite[Theorem 2]{BBM} and \cite[Theorem 1]{Davila}). One has that  for a set $E$ with finite perimeter in a neighborhood of $\Omega$, the local component of the fractional perimeter recovers, in the renormalized limit, the local perimeter of the set inside the domain $\Omega$, 
    \[ \lim_{s\nearrow 1} (1-s)P_s^L(E,\Omega)= \frac{\omega_{n-1}}{n-1} P(E,\Omega),\]
while we have that
    \[ \lim_{s\nearrow 1} (1-s)P_s^{NL}(E,\Omega)=\frac{\omega_{n-1}}{n-1} P(E,\partial \Omega),\]
    concluding that
    \eqlab{ \label{sto1} \lim_{s\nearrow 1} (1-s)P_s(E,\Omega)= \frac{\omega_{n-1}}{n-1} P(E,\bar \Omega).}
   Basically, in the limit, the far away data vanishes and the nonlocal component  concentrates on the boundary of the domain. 
   \bigskip
     
     The minimization problem is  the following.
     
     \begin{definition} \label{smin} 
Let $\Omega \subset \Rn$ be a bounded open set. Given $E_0:= E \setminus \Omega$, then $E$ is an $s$-{minimal set} in $\Omega$ with respect to $E_0$ if $P_s(E,\Omega)<\infty$ and
\[P_s(E, \Omega)\leq P_s(F,\Omega) \]
 for any $F$ such that
\[ F\setminus \Omega =E_0. \]
\end{definition}
     As in the classical case one obtains existence in the nonlocal framework by direct methods (check \cite[Theorem 3.2]{nms}, \cite[Theorem 1.8]{approxLuk}.)
     \begin{theorem}Let $\Omega \subset \Rn$ be an open set and let $E_0 \subset \Co \Omega$. There exist an $s$-minimal set in $\Omega$ with respect to $E_0$ if and only if there exists $F\subset \Rn$ with $F\setminus \Omega=E_0$ such that $P_s(F,\Omega)<\infty$.
     \end{theorem}
      In particular, asking $P_s(\Omega,\Rn)<\infty$ is enough to guarantee existence. Furthermore, interestingly, as a corollary of the previous theorem, local minimizer always exist (see \cite[Corollary 1.9]{approxLuk}). 
      
     As in the classical case again, it is much more involved to study the regularity of $s$-minimal sets.      
     Accordingly to \eqref{sto1}, for $s$ close to $1$, it is natural to expect properties similar to those of classical minimal surfaces (and this is proved in \cite{regularity}). For any $s\in (0,1)$, however, it is known that minimal surfaces are smooth up to dimension $2$ (thanks to \cite{SavV}). As a matter of fact, the best result to this day, following from \cite{regularity}, \cite{SavV} and \cite{bootstrap}, is the following.
     
     \begin{theorem} Let $s\in (0,1)$ be the fractional parameter,  $\Omega\subset \Rn$ be a bounded open set and $E$ be a $s$-minimal set. Then
    \begin{enumerate}
    \item $\partial E$ is smooth, up to a closed, singular set, of Hausdorff dimension at most $n-3$,
\item there exists $\eps_0 \in (0,1/2)$ such that for all $s\in (0,1-\eps_0)$,  $\partial E$ is smooth, up to a closed, singular set of Hausdorff dimension at most $n-8$.
\end{enumerate}
     \end{theorem}
    
    \subsection{Nonlocal minimal graphs}
The problem we look at in this subsection can be thought as the fractional version of Problem \ref{pb3}.

  \begin{problem}\label{pb4}
   Let $\Omega\subset \Rn$ be a bounded open set, and let $\varphi $ have integrable ``local tail''.  Find
   \[ \min \big\{ \mathcal F_s(u,\B) \; \big| \; u\in W^{s,1}(\Omega), u=\varphi \mbox{ in }  \Co \Omega \big\} .\]
   \end{problem}
   
Consider $F\subset \R^{n+1}$, that is the subgraph of some function $u$, that is
\[F:= Sg (u,\Omega)= \big\{ (x,x_{n+1}) \in \Omega\times \R\subset \R^{n+1}\; \big| \; x_{n+1}<u(x)\big\}.\]

In order to deal with nonlocal minimal graphs, one could take into consideration Remark \ref{nmg} and work in the geometric setting, thus trying to 
find the $s$-minimal graph which locally minimizes the $s$-perimeter in the class of subgraphs. This approach is motivated by a couple of observations: 
\begin{itemize}
\item according to \cite[Theorem 1.1]{graph}, if one considers $\Omega$ a bounded open set with $C^{1,1}$ boundary and the exterior data as a continuous subgraph in $\Co \Omega \times \R$, then the (local) minimizer of the $s$-perimeter is indeed a subgraph in $\Omega\times \R$ (and a local minimizer always exists according to \cite[Corollary 1.9]{approxLuk}),
\item an analogue of Point 2) of Remark \ref{nmg} is proved in \cite[Theorem 4.1.10]{Lucaphd} (and in the upcoming  paper \cite{LucaTeo}). If $F\setminus (\Co \Omega \times \R)$ is a subgraph, and $F\cap (\Omega \times \R)$  is contained in a cylinder, then the perimeter decreases if $F$ is replaced by a subgraph, built with a ``vertical rearrangement '' of the set $F$.  
\end{itemize}

In this setting, analogously to Point 3) in Remark \ref{nmg}, it is necessary to work with local minimizers, since the nonlocal part of the perimeter could give infinite contribution. 

However, remarkably in  \cite{Lucaphd} (and \cite{LucaTeo}), a very nice functional setting is introduced for the area of a graph, which is is equivalent to the perimeter framework in the following sense.

\begin{proposition} Let $\Omega\subset \Rn$ be a bounded open set  and $u\colon \Rn \to \R$ be a measurable function such that $u\in W^{s,1}(\Omega)$. If $u$ is a minimizer for $\mathcal F_s$, then $u$ locally minimizes $P_s(\cdot, \Omega\times \R)$ among sets with given exterior data $Sg(u,\Co \Omega)$. 
\end{proposition}

This $s$-fractional area functional is introduced in the next definition.
\begin{definition}\label{ars}
Let $\Omega \subset \Rn$ be a bounded open set, and let $u\colon \Rn \to \R$ be a measurable function. Then
\[ \mathcal F_s(u,\Omega):= \iint_{\R^{2n}\setminus (\Co\Omega)^2} \mathcal G_s \left( \frac{u(x)-u(y)}{|x-y|} \right) \frac{dx \, dy}{|x-y|^{n-1+s}},\]
where
\[ \mathcal G_s(t)= \int_0^t \left( \int_0^\tau \frac{d\rho} {(1+\rho^2)^{\frac{n+1+s}2} }\right) \, d\tau.\]
\end{definition}

The formula for the area functional is  motivated on the one hand, by the Euler-Lagrange equation for nonlocal minimal graphs. Namely, critical points of $\mathcal F_s$ are weak solutions of the Euler-Lagrange equation (see also Section \ref{el}). On the other hand, as mentioned previously, minimizing the area functional is equivalent  to minimizing the perimeter. It actually holds that
the local part of the area functional (that is, the interactions of $\Omega$ with itself) equals the perimeter of the subgraph of the function $u \in W^{s,1}(\Omega)$ (plus a constant term), and roughly speaking, the same relation holds between the nonlocal part of the area and that of the perimeter (see \cite[Lemma 4.2.7, 4.2.8]{Lucaphd}, \cite{LucaTeo}).

In order to have existence of Problem \ref{pb4} in $W^{s,1}(\Omega)$, one needs to ask a quite strong condition on the tail. This difficulty is surmounted by the authors of \cite{LucaTeo} by choosing a good notion of minimizer. We leave further explanations to the previously cited paper, mentioning that the existence result  is obtained in the following setting.  
Let $\mathcal O\subset \Rn$ be a given open set such that $\Omega $ is compactly contained in $\mathcal O$. Defining the ``local tail'' of a measurable function $\varphi \colon \Co \Omega\to \R$ as
 \[ \mbox{Tail}_s(\varphi, \mathcal O \setminus \Omega; x):= \int_{\mathcal O \setminus \Omega} \frac{|\varphi(y)|}{|x-y|^{n+s}} \, dy, \] 
 we can state
 the existence of solutions of Problem \ref{pb4} (see \cite[Theorem 4.1.3]{Lucaphd} and \cite{LucaTeo}).
 
  \begin{theorem}
 Suppose that $\mbox{Tail}_s(\varphi, \mathcal O \setminus \Omega; \cdot)\in L^1(\Omega)$ for $\mathcal O$ big enough depending on $\Omega$. Then there exists a unique minimizer of Problem \ref{pb4}.
  \end{theorem} 

As for regularity, combining results from \cite{Lucaphd,LucaTeo,Teocab} one has the following interior regularity theorem. 

 \begin{theorem}If $u\in W^{s,1}(\Omega)$ is a minimizer of $\mathcal F_s(\cdot, \Omega)$, then $u\in C^\infty(\Omega)$.
 \end{theorem} 

Boundary regularity of nonlocal minimal surfaces is a much more complicated and surprising story, and it gives a quite exhaustive answer to questions about the stickiness phenomena. A very recent result of \cite{sticks} establishes, at least in the plane, a dichotomy: either nonlocal minimal graphs are continuous across the boundary (and in that case, their derivatives are H\"{o}lder continuous), or they are not continuous, which equals to presenting stickiness.  
This result is contained in \cite[Corollary 1.3]{sticks}.  More precisely:

\begin{theorem}\label{dich}
Let $u\colon\R \to \R$, with $u\in C^{1,\frac{1+s}2}([-h,0])$ for some $h\in (0,1)$, be such that
$u$ is locally $s$-minimal for $\mathcal F_s(\cdot, (0,1))$.  Then 
\[ \overline{ \partial Sg(u) \cap \left( (0,1)\times \R\right) } \mbox{ is a closed, } C^{1,\frac{1+s}2} \mbox{ curve}.\]
Moreover, the following alternative holds:
\begin{enumerate}
\item either 
\[\lim_{x_1 \searrow 0} u(x_1)= \lim_{x_1\nearrow 0} u(x_1)\]
and 
\[u\in C^{1,\frac{1+s}2}([0,1/2]),\]
\item or
\[l= \lim_{x_1 \searrow 0} u(x_1)\neq \lim_{x_1\nearrow 0} u(x_1)\] 
and there exists $\mu>0$ such that 
		\[u^{-1} \in C^{1,\frac{1+s}2} ([l-\mu, l+\mu]).\]
\end{enumerate}
\end{theorem}

Notice that this theorem says that geometrically, the $s$-minimal graph is a $C^{1,\frac{1+s }2}$ curve in the interior of the cylinder, and up to the boundary. We further discuss Point 2) of this theorem in Section \ref{stick}.

\subsection{The fractional Euler-Lagrange equation}\label{el}

Classical minimal surfaces are characterized by the fact that at regular points, the mean curvature vanishes. This holds also in the fractional case, so we begin by introducing the fractional mean curvature (see \cite{abaty,nms}). Let $E\subset \Rn$ and $q\in \partial E$. Then
\[ \I_s[E](q) := P.V. \int_{\Rn} \frac{\chi_{\Co E}(x)-\chi_E(x)}{|x-q|^{n+s} } \, dx.\]
We will, for the sake of simplicity, omit the $P.V.$ in our computations.  

Just like for the $s$-perimeter, it holds that sending $s$ to $1$, the classical mean curvature appears. More precisely, let $E$ have $C^2$ boundary, then for any $q\in \partial E$ it holds that
\[ \lim_{s\nearrow 1} (1-s) \I_s[E](q)= \omega_{n-1} H[E](q),\]
where $H[E[(q)$ denotes the classical mean curvature at $q\in \partial E$, with the convention that balls have positive mean curvature. 

In the case of nonlocal minimal subgraphs $Sg(u)\subset \R^{n+1}$, one can give an explicit formula for the mean curvature, in dependence of the function $u$. Suppose for simplicity that we have a global minimal graph of $u\in C^{1,\alpha}(\Rn)$, which up to translations and rotations satisfies $u(0)=0, \nabla u (0)=0$. Then for $Q\in \partial Sg(u)$, (i.e. $u(q)=q_{n+1}$) one can write
\bgs{ \I_s[Sg(u)](Q) =&\; \int_{\Rn} \frac{\chi_{\Co Sg(u)}(X)-\chi_{Sg(u)}(X)}{|X-Q|^{n+1+s} } \, dX \\
					= &\; \int_{\R^{n}} dx'\int_{u(x)}^\infty \frac{dx_{n+1}}{(|x-q|^2 +|x_{n+1}-q_{n+1}|^2)^{\frac{n+1+s}2} }
					\\
					&\; -\int_{\R^{n}} dx'\int_{-\infty}^{u(x)} \frac{dx_{n+1}}{(|x-q|^2 +|x_{n+1}-q_{n+1}|^2)^{\frac{n+1+s}2} }
					\\
					=&\; \int_{\Rn} \frac{dx}{ |x-q|^{n+s} } \int_{ \frac{u(x)-q_{n+1} }{ |x-q| } }^\infty 
					 \frac{d\rho }{
					  (1+\rho^2)^{
					  \frac{n+1+s}2 
					  }
					  }
					 \\
					&\; - \int_{\Rn} \frac{dx'}{|x-q|^{n+s}} \int_{-\infty}^{\frac{u(x)-q_{n+1}}{|x-q|}} \frac{d\rho }{
					  (1+\rho^2)^{
					  \frac{n+1+s}2 
					  }
					  }
					  \\
					  =&\; 2\int_{\Rn} \frac{dx}{ |x-q|^{n+s} }  \int_0^{\frac{u(x)-q_{n+1}}{|x-q|}}  \frac{d\rho }{
					  (1+\rho^2)^{
					  \frac{n+1+s}2 
					  }
					  },					}
					where we have changed variables and have used symmetry. 
					Denoting
					\[ G_s(\tau)= \int_0^\tau \frac{d\rho }{
					  (1+\rho^2)^{
					  \frac{n+1+s}2 
					  }
					  },\]
					  recalling Definition \ref{ars} we notice that
					  \[ \mathcal G'_s(t)=G_s(t)\] 
					  which allows to prove, at least formally, that
					  \[ \frac{d}{d\eps}\Bigg|_{\eps=0} \mathcal F_s(u+\eps v)=0,\]
					  implies that, in a weak sense,
					  \[ \I_s[Sg(u)] =0. \]
					  This explains the connection between the fractional mean curvature operator and the functional formulation for the area operator in Definition \ref{ars}, introduced in \cite{LucaTeo}. 
					  \\
					  The formula for the mean curvature operator can be written also ``locally'', having $F\subset \R^{n+1}$ a set that is locally the graph of a function $u\in C^{1,\alpha}(B_r(q))$. Up to rotations and translations, and denoting for $r,h>0$
					  \[ K_r^h(Q):= B_r(q)\times (q_{n+1}-h,q_{n+1}+h),\]
					  one has that 
					  \eqlab{\label{af23}
					  \I_s[F](q)=  2\int_{B_r(q)} G_s\left(\frac{u(x)-u(q)}{|x-q|}\right) \frac{dx}{|x-q|^{n+s}} 
					  + \int_{\Rn \setminus K_r^h(Q)} \frac{\chi_{\Co Sg (u)}(X) -\chi_{Sg(u)}(X)}{|X-Q|^{n+1+s}} \, dX.
					  }
The reader can check \cite{regularity} where formula \eqref{af23} was first introduce, \cite{bootstrap} where the formula for the non-zero gradient is given, \cite{abaty,lukes} for further discussion on the mean curvature.
\smallskip

We give the Euler-Lagrange equation mentioned here above in the strong form, both in the interior and at the boundary of the domain. The following result, stated in a condensed form in \cite[Appendix B]{bucluk}, is a consequence of \cite[Theorem 5.1]{nms}, where the equation is given in the viscosity sense, \cite{obss,bootstrap} where regularity is settled, and \cite{graph}, where the authors go from the viscosity to the strong formulation. 
\begin{theorem}\label{el} Let $\Omega\subset \Rn$ be an open set and let $E$ be locally $s$-minimal in $\Omega$. 
\begin{enumerate}
\item If $q\in \partial E$ and $E$ has either an interior or an exterior tangent ball at $q$, then there exists $r>0$ such that $\partial E \cap B_r(q)$ is $C^\infty$ and
\[ \I_s[E](x)=0 \quad \mbox{ for any } x\in \partial E\cap B_r(q).\]
In particular,
\[ \I_s[E](x) =0 \quad \H^{n-1}-\mbox{a.e. for } x\in \partial E \cap \Omega.\]   
\item If  $q \in \partial E \cap \partial \Omega$ and $\partial \Omega$ is $C^{1,1}$ in $B_{R_0}(q)$ for some $R_0>0$,  and $B_{R_0}(p)\setminus \Omega\subset \Co E$, then
\[ \I_s[E](q) \leq 0.\] 
Moreover, if there exists $R<R_0$ such that
\[ \partial E \cap \left(\Omega \cap B_r(q)\right) \neq \emptyset \qquad \mbox{ for any } r<R\]
then
\[ \I_s[E](q) = 0.\]  
\end{enumerate}
\end{theorem}

This theorem provides the Euler-Lagrange equation almost anywhere in the interior of the domain $\Omega$ (at all regular points), and at the boundary of $\Omega$ with smooth boundary, as long as, roughly speaking, $E$ detaches from the boundary of $\Omega$ towards the interior, or $\partial E$ coincides with $\partial \Omega$ near the point $q$.

\section{The stickiness phenomena for nonlocal minimal surfaces}\label{stick}

In the nonlocal setting, the stickiness phenomena is typical. The situation drastically changes with respect to the classical objects since even in convex domains and with smooth exterior data, the $s$-minimal surface may attach to the boundary of the domain. 
A first example is given in \cite[Theorem 1.1]{boundary} showing stickiness to half-balls. We look for a nonlocal minimal set in a ball, having as exterior data a half-ring around that ball. A small enough radius of the ring will lead to stickiness. Precisely: 

\begin{theorem}\label{halfball}
For any $\delta>0$, denote
\[ K_\delta := \left(B_{1+\delta}\setminus B_1\right) \cap\{ x_n<0\},\]
and let $E_\delta$ be $s$-minimal for $P_s(\cdot, B_1)$ with $E\setminus B_1=K_\delta$. 
There exists $\delta_0:=\delta_0(n,s)>0$ such that for any $\delta \in (0,\delta_0]$ we have that
\[ E_\delta=K_\delta.\] 
\end{theorem}

Not only does stickiness happen in  unexpected situations, what is more is that small perturbations of the exterior data may cause stickiness. We describe this phenomena with the example given in \cite[Theorem 1.4]{boundary}. It is well known that the only $s$-minimal set with exterior data given by the half-plane is the half-plane itself. But surprisingly, flat lines are ``unstable'' $s$-minimal surfaces in the following sense. Changing slightly the exterior data by adding two compactly contained ``bumps'', the $s$-minimal surface in the cylinder sticks to the walls of the cylinder, for a portion which is comparable to the height of the bumps.  The exact statement is the following.

 \begin{theorem}
 Fix $\eps_0>0$ arbitrarily small. Then there exists $\delta_0:=\delta_0(\eps_0)>0$ such that for any $\delta\in(0,\delta_0]$ the following holds true.
 Consider
 \[ H= \R \times (-\infty,0) \qquad F_- =(-3,-2)\times[0,\delta), \qquad F_+=(2,3)\times[0,\delta),\]
 and
 \[ F\supset H\cup F_-\cup F_+.\]
 Let $E$ be the $s$-minimal set in $(-1,1)\times \R$ among all sets such that $E=F$ outside of $(-1,1)\times \R$. Then
 \[ E\supseteq (-1,1) \times (-\infty, \delta^{\frac{2+\eps_0}{1-s}}).\]
 \end{theorem} 

The proof of this theorem is very interesting in itself, carried out by building a suitable barrier from below. 

As a matter of fact, taking into account the dichotomy in Theorem \ref{dich}, it is clear that this unstable behavior appears to be typical. This is the case: even in the plane, if we start with a $s$-minimal surface which is continuous across the boundary, it is enough to perturb slightly the exterior data in order to get stickiness. Indeed, consider $v\colon \R \to \R$ smooth enough, fixed outside of the interval $(0,1)$, which plays the role of the exterior data, and let $u\colon \R \to \R$, $s$-minimal with respect to $v$, be continuous across the boundary. Then smoothly perturbing $v$ outside of the cylinder will produce a $s$-minimal graph which sticks to the cylinder. 
 This generic behavior is better explained in \cite[Theorem 1.1]{sticks}. 

\begin{theorem}
Let $\alpha \in (s,1)$,  the function $v\in C^{1,\alpha}(\R), $
and
$ \varphi \in C^{1,\alpha}(\R)$ non-negative and not identically zero, such that $ \varphi=0 \mbox{ in } (-d,d+1) \mbox{ for some } d>0.$ Consider then $u\colon \R \times [0,\infty) \to \R$ such that
\[ u(x_1,t)= v(x_1)+t\varphi(x_1), \qquad \,  t\geq 0,\,  x_1\in \R\setminus (0,1)\]
and suppose that
the set
\[ E_t= \big\{ (x_1,x_2)\in \R^2 \, \big| \, x_2<u(x_1,t)\big\} \]
is locally $s$-minimal in $(0,1)\times \R$. 
Assume that
\[ \lim_{x_1\searrow 0} u(x_1,0)= v(0).\]
Then for any $t>0$
\[ \limsup_{x_1\searrow 0} u(x_1,t)>v(0).\]  
\end{theorem}

\section{Complete stickiness in highly nonlocal regimes}

A very nice example of complete stickiness, that is when the minimal surface attaches completely to the boundary of the domain,  was recalled in Theorem \ref{halfball}. 
On the one hand, complete stickiness depends on how ``large''  the exterior data is. On the other hand, fixing the exterior data, we obtain complete stickiness  for $s$ small enough.
 Indeed, as $s$ gets smaller, the nonlocal contribution prevails and the effects are quite surprising. In this section,  we sum up  some results from the literature related to highly nonlocal regimes, and provide examples of complete stickiness both for nonlocal minimal sets and graphs.

To describe the ``purely nonlocal contribution'', one makes use of the set function introduced in \cite{asympt1}
\eqlab{ \label{alpha} 
\alpha(E)= \lim_{s \searrow 0} s  \int_{\Co B_1}\frac{ \chi_E(x)} {|x|^{n+s}} \, dx.}
As \cite[Examples 2.8, 2.9]{asympt1} show, it is possible to have smooth sets (hence with finite $s$-perimeter for any $s$) for which the limit in \eqref{alpha} does not exist. In this case, neither $\lim_{s\searrow 0} s P_s(E,\Omega)$ exists, since the two limits are intrinsically connected. Whenever this happens, one can use $\limsup$ and $\liminf$ as in \cite{bucluk}.  For simplicity, we use however $\alpha$ as defined in \eqref{alpha}, and notice that the results in this section hold for $\limsup  (\liminf)$ instead of the limit, whenever the limit does not exist.

The fact that this set function well describes the behavior of the perimeter as $s$ goes to $0$ is given in \cite[Theorem 2.5]{asympt1}.

\begin{theorem}
Let $\Omega\subset \Rn$ be a bounded open set with $C^{1,\gamma}$ boundary for some $\gamma \in (0,1)$.  
Suppose that $P_{s_0}(E,\Omega)$ is finite for some $s_0\in (0,1)$. Then
\eqlab{ \label{pers0} \lim_{s\searrow 0} s P_s(E,\Omega)= \alpha(\Co E) |E \cap \Omega| + \alpha (E) |\Co E \cap \Omega|.}
\end{theorem}
If one goes back to \eqref{pss}, one gets that the local contribution  completely vanishes in the limit
\[ \lim_{s\searrow 0} s P_s^{L}(E,\Omega)=0.\]
On the other hand, in the limit, the nonlocal part gives a combination of the purely nonlocal contribution, expressed in terms of the function $\alpha$, and the Lebesgue measure of the set (or its complement) in $\Omega$. Recalling also the limit as $s\nearrow 1$ in \eqref{sto1}, one could say that in some sense, the fractional perimeter interpolates between the perimeter of the set and its volume. It is even clearer if we take, for example, a set $E$ bounded, with finite perimeter, contained in $\Omega$. Then \eqref{pers0} and \eqref{sto1} give that
\[ \lim_{s\searrow 0} s P_s(E,\Omega) =\omega_n |E|  \]
and
\[ \lim_{s \nearrow 1} (1-s)P_s(E,\Omega)= \frac{\omega_{n-1}}{n-1} P(E,\Omega).\] 

A second element describing purely nonlocal regimes comes from the mean curvature operator. What we discover is that, as $s$ decreases towards zero, in the limit the mean curvature operator forgets any local information it had detained on the local geometry of the set, and measures only the nonlocal contribution of the set. More precisely
\eqlab{ \label{mcs0}\lim_{s\searrow 0}s \I_s[E](p)= \omega_n -2 \alpha (E),}
for any $p\in \partial E$ and whenever $\partial E$ is $C^{1,\gamma}$ around $p$, for some $\gamma \in (0,1]$.  

We provide a few more details on the set function $\alpha(E)$, which are useful in the sequel. 
Denote for $q\in \Rn$ and $R>0$
\eqlab{ \label{alphaeq}
\alpha_s (E,R,q) = \int_{\Co B_R(q) } \frac{\chi_E(x)}{|x-q|^{n+s} }\, dx.}
Then it holds that
\[ \lim_{s \searrow 0} s\alpha_s(E,R,q)= \alpha(E).\] 
In particular, this says that $\alpha$ represents indeed the contribution from infinity, as it does not depend neither on the fixed point $q\in \Rn$, nor on the radius we pick. So, to compute the contribution from infinity of a set it is enough to compute its weighted measure outside of a ball of any radius, centered at any point. For more details and  examples, check \cite[Section 4]{bucluk}. We just recall here  a couple of examples, which are therein explained: the contribution from infinity
\begin{itemize}
\item of a bounded set is zero,
\item of a cone is given by the opening of the cone,
\item of a slab is zero,
\item of the supergraph of a parabola is zero,
\item of the supergraph of $x^3$ in $R^2$ is $\pi$,
\item of the supergraph of a bounded function is $\omega_n/2$.
\end{itemize}
 
\subsection{Complete stickiness}

We start this subsection  with an example. As we have already mentioned, the only $s$-minimal set having as the half-space as exterior data is the half-space itself, for any value of $s$. On the other hand, let us try to understand what happens if we minimize the perimeter in $B_1\subset \R^2$, using the first quadrant of the plane as exterior data. As (\cite[Theorem 1.3]{boundary}) shows, there exists some small $s_0$ such that for all   $s\in (0,s_0)$ the $s$-minimal surface sticks to $\partial B_1$, and the $s$-minimal set is exactly the first quadrant  of the plane, deprived of its intersection with $B_1$.  
This example still holds if, instead of the ball, one picks a domain $\Omega$, bounded, with smooth boundary and takes as the exterior data the whole half-plane, deprived of some small cone, at some distance from $\Omega$. For simplicity, we give an example that one can keep in mind, before we introduce the main theorem of the section.

\begin{example}\label{exxx}
Let for any given $h\geq 1$ and  $\theta  \in (0,\pi/2)$
\bgs{ \Sigma := \Big\{ (x_1,x_2)\in \R^2 \; \Big|  x_2\geq \Big((x_1-h)\tan \theta \Big)_+ \Big\}}
and let $E_0:= \Sigma \setminus B_1$.  
Then there exists $s_0 >0$ such that for any $s \in (0,s_0)$, the set $E_s$ that minimizes $P_s(\cdot, B_1)$ with respect to $E_0$, is empty inside $B_1$, 
 or in other words
 \[ E_s =\Sigma\setminus B_1.\]
\end{example}

\begin{proof}[Sketch of proof]
 We argue by contradiction and suppose that there is some boundary of $E$ inside $\Omega$. We follow the next steps.
\begin{enumerate}
	\item Step 1. We prove that, if there exists an  exterior tangent ball at a point on the boundary of $E\cap \bar B_1$, of some suitable (uniform) radius, the fractional mean curvature of $E$ at that point is strictly positive. 
	\item Step 2. We prove that there exists some ball, compactly contained in $B_1$, which is exteriorly tangent to the boundary of $E$. 
	\item Step 3. We obtain a contradiction by comparing Step 1 with the Euler-Lagrange equation (that holds, thanks to Step 2, check Theorem \ref{el}).
	\end{enumerate}
\smallskip

\noindent \textbf{Step 1}. 
We have set out to prove that, if there exists an exterior tangent ball at $q\in \partial E\cap \bar B_1$, there exists $\tilde C>0$ such that
\[ \I_s[E](q) = \int_{\R^n} \frac{ \chi_{\Co E}(x) -\chi_E(x)}{|x-q|^{n+s}} \, dx \geq \tilde C.\]
Let $\delta$ be a radius (that will be chosen as small as we want in the sequel), and $p\in B_1$ such that $B_\delta(p)$ is  compactly contained in $B_1$, exterior tangent to $\partial E$ at $q$, that is
\[ B_\delta(p) \subset\Co E\cap B_1, \quad q \in \partial E \cap \partial B_\delta(p).\]
Denote $p'$ as the point symmetric to $p$ with respect to $q$,  
\[ D_\delta :=B_\delta(q) \cup B_\delta(p'), \]
$K_\delta$ as the convex hull of $D_\delta$ and
\[ P_\delta :=K_\delta \setminus D_\delta.\]
Let $R>4$ be as large as we want, to be specified later on.

We split the integral into four different parts and estimate each one.

\begin{enumerate}
\item The contribution in $D_\delta$ is non-negative, since $E$ covers ``less'' of $D_\delta$ than of its complement, i.e.
\[ \chi_{\Co E\cap D_\delta}\geq \chi_{E \cap D_\delta},\]
hence
\[ \int_{D_\delta} \frac{ \chi_{\Co E} (x) -\chi_{E}(x)}{|x-q|^{n+s} } \geq 0.\]
\item The contribution on $P_\delta$ is bounded from below thanks to \cite[Lemma 3.1]{graph},
\[ \int_{P_\delta} \frac{ \chi_{\Co E} (x) -\chi_{E}(x)}{|x-q|^{n+s} } \geq -C_1\delta^{-s}.\]
\item As for the contribution in $B_R(q) \setminus K_\delta$, we have that
\bgs{
  &\bigg|\int_{B_R(q) \setminus K_\delta} \frac{ \chi_{\Co E} (x) -\chi_{E}(x)}{|x-q|^{n+s} }\bigg| 
  \leq 
 \bigg |\int_{B_R(q) \setminus B_\delta(q)} \frac{ \chi_{\Co E} (x) -\chi_{E}(x)}{|x-q|^{n+s} } \bigg|
  \\
  \leq &\; \omega_n \int_{\delta}^R \rho^{-1-s} \, d\rho =\omega_n \frac{ \delta^{-s}-R^{-s}}s .}
  \item We prove that the contribution of $\Co B_R(q)$ is bounded by
  \[ \int_{\Co B_R(q)} \frac{ \chi_{\Co E} (x) -\chi_{E}(x)}{|x-q|^{n+s} } \geq \frac{C (\theta)R^{-s}}{s},\]
  for some constant $C(\theta)\in (0,\omega_n/2)$, in particular independent on $q$. 
  
\end{enumerate}

Of course, $\omega_n$ is actually $\omega_2$, but we keep the above formulas in this general from since the estimates  hold in any dimension.

 Putting the four contributions together, our goal is to obtain that
 \[ s\I_s[E](q) \geq \left(C(\theta) +\omega_n\right) R^{-s}- \delta^{-s}( C_1s +\omega_n) \geq \frac{C(\theta)}{8}>0.\]
 Since $R^{-s} \nearrow 1$ as $s\searrow 0$, there exists $s$ small enough such that
 \[ C(\theta) R^{-s} \geq \frac{C(\theta)} 2, \quad \omega_n R^{-s} \geq \omega_n -\frac{C(\theta)}4 , \quad C_1 s \leq \frac{C(\theta)}{16} \]
 thus
\[ s\I_s[E](q) \geq  \frac{C(\theta)}4 + \omega_n  -\delta^{-s}\left(\omega_n +\frac{C(\theta)}{16} \right)\geq \frac{C(\theta)}8,\]
if and only if
 \eqlab{ \label{dss}\delta \geq   e^{\frac{-1}{s } \log \frac{8\omega_n+C(\theta)}{8\omega_n +C(\theta)/2} }:=\delta_s.}
Notice that $\delta_s <1$, hence for any $s\in(0,\sigma)$ taking $ \delta>\delta_\sigma,$
\[ \qquad \delta^{-s} <\delta^{-\sigma}<\delta_{\sigma}^{-\sigma},\]
hence for any radius greater than $\delta_\sigma$ the $s$-curvature will remain strictly positive for any $s<\sigma$.
 We can conclude that
 there exists $\sigma$ such that, having at  $q$ an exterior tangent ball of radius (at least) $\delta_\sigma$, implies that
 \[ s\I_s[E](q) \geq \frac{C(\theta)}8 >0 \quad \mbox{ for all } \; s\leq \sigma.\]
 \smallskip

\noindent \textbf{Step 2}.   To carry out Step 2, we prove that there exists an exterior tangent ball to $\partial E$, compactly contained in a ball slightly smaller than $B_1$.
We denote 
	\[ B_1^+ = B_1\cap \{x_2>0\} , 
		\quad B_1^- =B_1 \cap \{ x_2<0\}, 
		.
	\] 
First of all, we notice by comparison with the plane, that
\[ B_1^-  \subset \Co E.\]
Otherwise, we start moving upwards the semi-plane $\{ x_2 \leq 2\}$ until we first encounter $\partial E \cap \bar B_1^-$ at  $p=(p_1,p_2)$. Since
\[ \Co E \supset \Co \{x_2>p_2\}, \qquad E\subset  \{x_2>p_2\}\] it holds that 
\[ \I_s[E](p) = \I_s[E](p)-\I_s[\{x_2>p_2\}] (p)\geq 0,\] 
and since $E$ is minimal, it holds in the strong sense that 
\[ \I_s[E](p) \leq 0. \]
This would imply that $E=\{x_2<p_2\}$ by the maximum principle (see \cite[Appendix B]{bucluk}), which is false.\\
For some $r_0>0$ and $s$ small enough (notice that $\delta_s \searrow 0$ as $s\searrow 0$, see \eqref{dss}), and $x\in B_1^{-}$, consider $\delta_s<\delta<r_0/4$ such that
\[ B_\delta(x) \subset B_{1-r_0/2}^- \subset \Co E .\]
 We remark that for a domain $\Omega$ with $C^2$ boundary, $r_0$ is chosen to be  such that
 \eqlab{ \label{r0} \mbox{ the set  } \; \big\{ x\in \Omega \; \big| \; d(x,\partial \Omega)\leq r_0 \big\} \mbox{ still has } C^2 \mbox{ boundary}
 } 
(check \cite[Appendix A.2]{bucluk}, \cite[Appendix B]{MinGiusti} for instance). 
Suppose now by contradiction that $E$ is not empty inside $B_{1-r_0/2}$, hence
\[|E \cap B^+_{1-r_0/2}|>0, \quad \mbox{ in particular} \quad \exists\;  y\in E\cap B^+_{1-r_0/2} .\]

We consider the segment connecting $x$ and $y$ inside $B_{1-r_0/2}$, and we move the ball of radius $\delta$ along this segment starting from $x$, until we first hit the boundary of $E$. We denote by $q$ the first contact point (for a more detailed discussion, see \cite[Lemma A.1]{bucluk}), i.e. for $p\in B_{1-r_0/2}^+$
\[ q\in \partial E \cap \partial B_\delta(p), \qquad B_\delta(p)\subset \Co E.\]
 
\smallskip

 \noindent \textbf{Step 3}. 
 Since at $q$ there exists an exterior tangent ball of radius $\delta$, we use the Euler-Lagrange equation in the strong form and have that 
 \[ \I_s[E](q)=0.\]
 This provides a contradiction  with Step 1, and it follows that 
 \[ |E\cap B_{1-r_0/2} | =0.\]
 Now it is enough to ``expand'' $B_{1-r_0/2}$ towards $B_1$. If there is some of $E$ in the annulus $B_1 \setminus B_{1-r_0/2}$, one can find an exterior tangent ball at $\partial B_{1-\rho} \cap \partial E$ for some $\rho \in (0,r_0/2)$ and use again the fact that the curvature is both strictly positive and equal to zero to obtain a contradiction. This would conclude the proof.
 
 \medskip
 
 It remains to prove that for $q\in \partial E \cap \bar B_1$
  \[ \int_{\Co B_R(q) } \frac{ \chi_{\Co E} (x) -\chi_{E}(x)}{|x-q|^{2+s} } \geq \frac{C (\theta)R^{-s}}{s},\]
  for some constant $C(\theta)$ not depending on $q$. 
   We do this with a geometric argument. We want to build a parallelogram of center $q$, and take $R$ as large as we need, such as to have the parallelogram in the interior of $B_R(q)$. Then we use symmetry arguments to obtain the conclusion.
   
    We build our parallelogram in the following way, check Figure \ref{par}. We denote 
   \[ l_1= (x_1-h) \tan \theta\]
   and draw through $q$ the parallel to the bisecting line of the angle complementary to $\theta$. We call $p$ the intersection between this parallel line and $l_1$, and $p'$ the point symmetric to $p$ with respect to $q$, that sits on this parallel line. We draw through $p,p'$ two lines parallel to the axis $Ox$. The parallelogram we need is formed by the intersections of these last drawn parallels to $Ox$, $l_1$ and the parallel to $l_1$ through $p'$. We choose $R$ such that this parallelogram stays in the interior of $B_R(q)$, remarking that $R$ depends only on $\theta,h $, and we can make this choice independent on $q\in \bar B_1$. In particular, one can take
   \[ R:=\max\{ \max_{x\in \bar B_1}d (x,l_1) \cot \frac{\theta}4, 4\}. \]
    \begin{center}
\begin{figure}[htpb]
	\centering
	\includegraphics[width=0.9\textwidth]{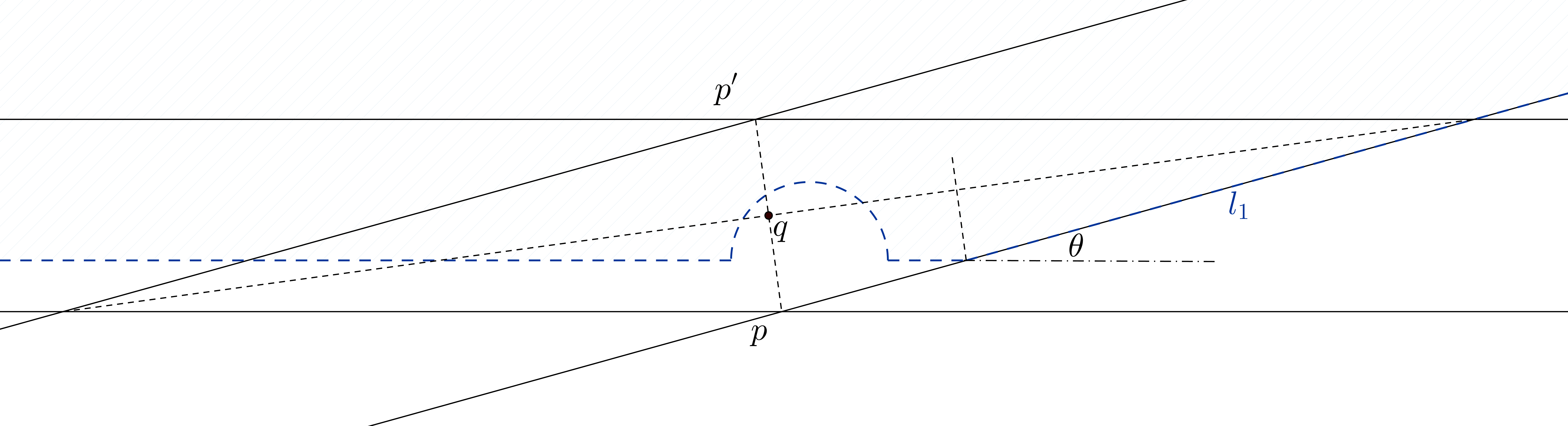}
	\caption{The geometric construction in Example \ref{exxx}}   
	\label{par}
\end{figure} 
	\end{center}
   This ensures that both $B_1$ and the parallelogram we built  are in $B_R(q)$. We identify six ``corresponding'' regions, which by symmetry produce some nice cancellations. Not to introduce heavy notations, the reader can check directly Figure \ref{par1}. 
  \begin{center}
     \begin{figure}[htpb]
	\centering
	\includegraphics[width=0.9\textwidth]{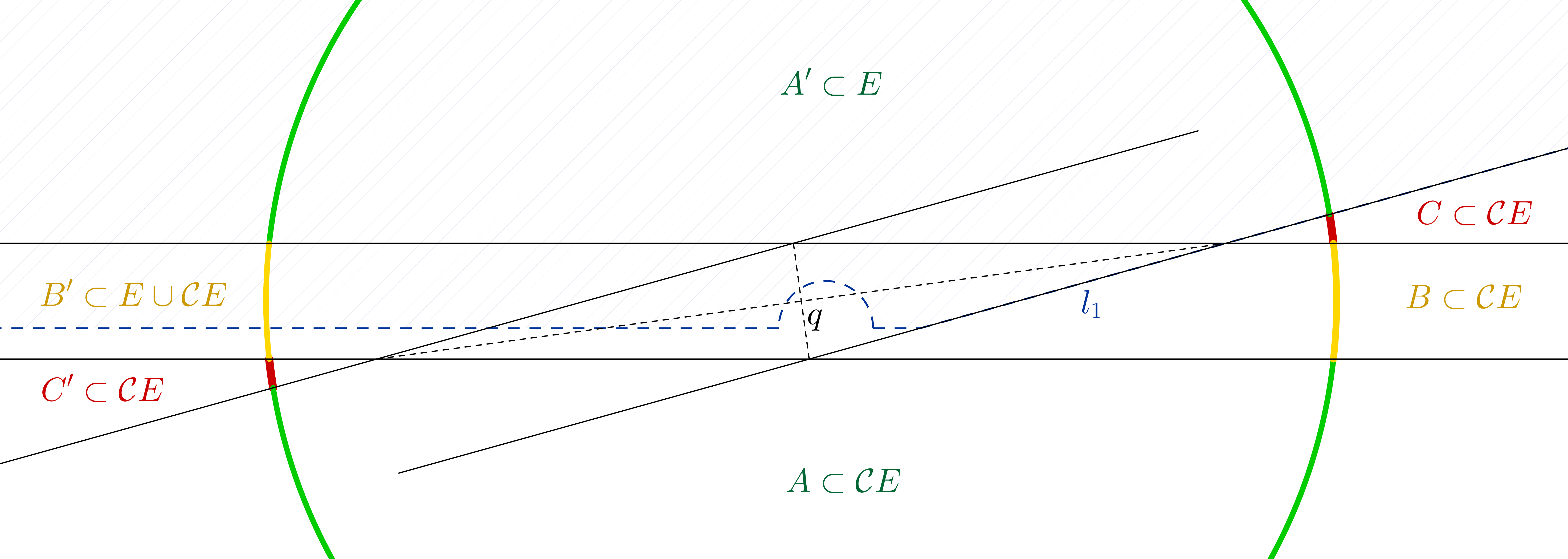}
	\caption{The geometric construction in Example \ref{exxx}}   
	\label{par1}
\end{figure} 
	\end{center}
	Notice that
   \bgs{ & A\subset \Co E, \quad A' \subset E,   
   \\
   & B \subset \Co E, \quad B' \subset E \cup \Co E,
   \\
   & C\cup C' \subset \Co E}
   and accordingly we have that
   \bgs{\int_{\Co B_R(q) } \frac{ \chi_{\Co E} (x) -\chi_{E}(x)}{|x-q|^{2+s} } \, dx=&\; 
   \left(\int_{ A\cup A'} + \int_{B\cup B'} + \int_{C \cup C'} \right) \frac{ \chi_{\Co E} (x) -\chi_{E}(x)}{|x-q|^{2+s} } \, dx
   \\ \geq &\;  2 \int_C \frac{dx}{|x-q|^{2+s} }.
   }
   Now $C$ contains a cone $\mathcal C_{\theta}(q)$ centered at $q$, of opening $\gamma:=\gamma(\theta)$, independent on $q$.
  
 In particular (see Figure \ref{gg}) we have that
   \[ \frac{\gamma}2 =\frac{\pi}2 -\alpha-\frac{\pi-\theta}2 \geq \frac{\theta}2 -\frac{\theta}4 =\frac{\theta}4,\]
   given that
   \[ \cot \alpha = \frac{R}{d(q,l_1)} \geq \frac{\max_{x\in \bar B_1} d(x,l_1) \cot \frac{\theta}4}{d(q,l_1)}\geq \cot \frac{\theta}4. \]  
    \begin{center}
\begin{figure}[htpb]
	\centering
	\includegraphics[width=0.65\textwidth]{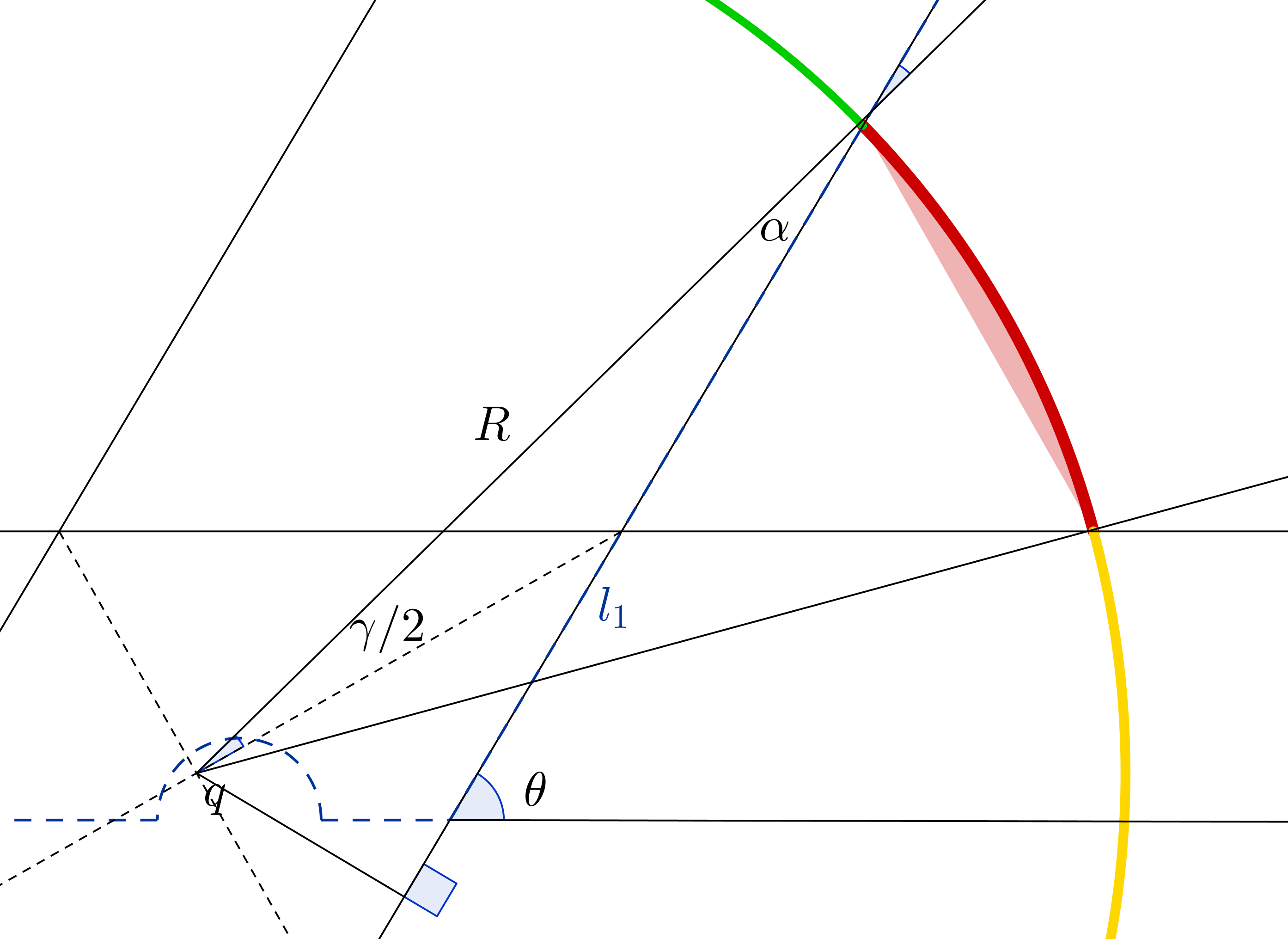}
	\caption{The small cone $\mathcal C_\theta(q)$ in Example \ref{exxx}}   
	\label{gg}
\end{figure} 
	\end{center}
     Passing to polar coordinates, it follows that
   \[  \int_C \frac{dx}{|x-q|^{2+s} } \, dx  \geq \int_{\mathcal C_{\theta}(q)} \frac{dx}{|x-q|^{2+s} } \, dx =\gamma  \frac{R^{-s}}{s}\geq \frac{\theta}2 \frac{R^{-s}}s.\]
   This concludes the sketch of the proof. 
\end{proof}

The reader may wonder if this behavior depends on the particular geometry of the sets involved. The answer is no, and actually it only matters that the exterior data occupies, at infinity, less than half the space, or mathematically written
\[ \alpha(E_0) <\frac{\omega_n}{2}.\]

Intuitively, one can try to understand why this is to be expected. Let us first check
\eqref{pers0}, and re-write it as
\[ \lim_{s\searrow 0} s P_s(E,\Omega) = \alpha(E_0)|\Omega| + \left(\omega_n -2\alpha(E_0)\right) |E \cap \Omega|.\] 
In broad terms, minimizing the perimeter for $s$ small reduces to minimizing
 $(\omega_n -2\alpha(E_0))|E\cap \Omega|$. Hence if 
 \[ \alpha(E_0)<\omega_n/2\] 
 the best choice to select the minimal set is to take $E\cap \Omega =\emptyset$ (whereas, for $\alpha(E_0)>\omega_n/2$, $E\cap \Omega=\Omega$ would be the right choice). 
We notice also that if $\alpha(E_0)=\omega_n/2$, we do not get any information at this point.
 
Another element that can help, and that further strengthen the intuition, is the asymptotic behavior of the fractional mean curvature \eqref{mcs0}.
Suppose now that $\alpha(E_0)<\omega_n/2$. Then, given the continuity of the fractional mean curvature in $s$ (see \cite[Section 5]{bucluk}),  from \eqref{mcs0} for $s$ small enough  it follows that 
\[ \I_s[E](x) >0,\]
(and this holds for any set $E$ such that $E\setminus \Omega=E_0$, not only for $s$-minimal sets).
This strict positivity of the mean curvature comes very handy when one compares it with the Euler-Lagrange equation recalled in Theorem \ref{el}.  If there exists an exterior (or interior) tangent ball to the minimal surface $\partial E$, then 
\[ \I_s[E](x) =0.\]
This would provide a contradiction at all (smooth) points on the boundary of the minimal set, inside the domain $\Omega$, and would show that there cannot be any boundary of $E$ inside $\Omega$. 

This informal discussion can be set in the following theorem (see \cite[Theorem 1.7]{bucluk}).

\begin{theorem}\label{thm}
Let $\Omega \subset \Rn$ be a bounded and connected open set with $C^2$ boundary and let $E_0 \subset \Co \Omega$ be given such that
\[ \alpha(E_0)<\frac{\omega_n}2.\]
Suppose that $E_0$ does not completely surround $\Omega$, i.e., there exists $M>0$ and $x_0\in \partial \Omega$ such that
\eqlab { \label{bmm} B_M(x_0)\cap \Co \Omega \subset \Co E_0.}
Then there exists $s_0\in (0,1/2)$ such that for all $s<s_0$, the corresponding $s$-minimal surface sticks completely to the boundary of $\Omega$, that is
\[ E\cap \Omega=\emptyset.\]  
\end{theorem}

\begin{proof}[Sketch of the proof]

We follow the proof of Example \ref{exxx}, with some additional difficulties.

\noindent \textbf{Step 1}. In order to carry out Step 1, we split the integral into the four components, exactly as we did in Example \ref{exxx}. Let $\delta$ be a radius (that will be chosen as small as we want in the sequel), and $p\in \Omega$ such that $B_\delta(p)$ is  compactly contained in $\Omega$, exterior tangent to $\partial E$, that is
\[ B_\delta(p) \subset\Co E\cap \Omega, \quad q\in \partial E \cap \partial B_\delta(p).\]
Let $R>4$ be as large as we wish. We observe that the estimates in 1), 2) and 3) stay exactly the same. It only remains to prove 4), and actually we notice that 
\bgs{ \int_{\Co B_R(q)} \frac{\chi_{\Co E} (x)- \chi_E(x)}{|x-q|^{n+s}} \, dx = &\;
 \int_{\Co B_R(q)} \frac{1- 2 \chi_E(x)}{|x-q|^{n+s}} \, dx
 \\
 =&\;\frac{\omega_n R^{-s}}s - \alpha_s(E,R,q),}
recalling \eqref{alphaeq}. 
Then it follows that
\bgs{
		s\I_s[E](q) \geq \omega_n R^{-s}- \delta^{-s}( C_1 s +\omega_n) + \omega_n  R^{-s}- 2s\alpha_s(E,R,q) .}
  Now
  \[ \lim_{s\searrow 0}\left( \omega_n R^{-s} -2s\alpha_s(E,R,q) \right)= \omega_n-2\alpha(E):=C(E).\]
    The computations follow exactly as in the proof of Example \ref{exxx}, with $C(E)$ instead of $C(\theta)$.   
  Notice also that, in case $E$ is a cone, $\alpha(E)$ is exactly the opening of the cone (hence, $\alpha(\Sigma)=2\theta$).  
  
   Therefore there exists $\sigma$ such that, for all $s\leq \sigma$, having at  $q$ an exterior tangent ball of radius (at least) $\delta_\sigma$, implies that
 \eqlab{ \label{mmma} s\I_s[E](q) \geq \frac{C(E)}4 >0.}
  \medskip
  
\noindent \textbf{Step 2}. 
In order to prove Step 2, we need to fit  a ball of suitable small radius inside $\Omega	\cap \Co E$. 

We define $r_0$ as in \eqref{r0}, and $\sigma$ small enough such that
\[ \delta_\sigma<\delta \leq \frac14 \min\{M,r_0\}. \]
Since $\delta>\delta_{\sigma}$, \eqref{mmma} holds.

Denote by $\nu_\Omega(x_0)$ the exterior normal to $\partial \Omega$ at $x_0\in \partial \Omega$. ``Taking a step'' of length $\delta$ away from the boundary of $\Omega$ inside the ball $B_M(x_0)$, in the direction of the normal, reaching $x_1$, we have that $B_\delta(x_1) \subset B_M(x_0)\cap \Co \Omega\ \subset \Co E$. We want to ``move'' this ball along the normal towards the interior of $\Omega$, until we reach $x_2$, the  point on the normal at distance $r_0$ from the boundary of $\Omega$. We can exclude an encounter with $\partial E$, both on the boundary of $\Omega$ and inside of $\Omega$, since in both cases we have the Euler-Lagrange equation and Step 1, which provide a contradiction. Thus, denoting
\[ \Omega_{-r_0/2}:= \Big\{ x\in \Omega \; \Big| \; d(x,\partial \Omega)=\frac{r_0}2\Big\},
\]
we have that
   \[B_\delta(x_2) \subset \Omega_{r_0/2} \cap \Co E.\]

Now, if the boundary of $E$ lies inside $\Omega_{-r_0/2}$, we pick $p\in E\cap \Omega_{-r_0/2}$ and slide the ball $B_\delta(x_2)$ along a continuous path connecting $x_2$ with $p$. At the first contact point on $\partial E \cap \partial B_\delta(\bar x)$, with $\bar x$ lying on the continuous path between $x_2, p$, we obtain a contradiction from Step 1 and the Euler-Lagrange equation. We obtain the same contradiction by ``enlarging'' $\Omega_{-r_0/2}$, since, at the first contact point the ball $B_{\frac{r_0}4}$ provides a tangent exterior ball to $\partial E \cap \Omega_{-\rho}$, for some $\rho \in (0,r_0/2)$  We obtain that $E\cap \Omega =\emptyset$,   concluding the sketch of the proof.
\end{proof}

Of course, the analogue holds for the data that occupies, at infinity, more than half the space. In that case, the result is as follows.

\begin{theorem}\label{thm}
Let $\Omega \subset \Rn$ be a bounded and connected open set with $C^2$ boundary and let $E_0 \subset \Co \Omega$ be given such that
\[ \alpha(E_0)>\frac{\omega_n}2.\]
Suppose that $\Co E_0$ does not completely surround $\Omega$, i.e., there exists $M>0$ and $x_0\in \partial \Omega$ such that
\eqlab { \label{bmm} B_M(x_0)\cap \Co \Omega \subset E_0.}
Then there exists $s_0\in (0,1/2)$ such that for all $s<s_0$, the corresponding $s$-minimal surface sticks completely to the boundary of $\Omega$, that is
\[ E\cap \Omega=\Omega.\]  
\end{theorem}

On the other hand, if 
\[\alpha(E)=\frac{\omega_n}2,\]
 neither \eqref{pers0} nor \eqref{mcs0} provide any additional information, 
since we get that
\[ \lim_{s\searrow 0} sP_s(E,\Omega)= \frac{\omega_n}2|\Omega|\]
and that for any $q\in \partial E$
\[ \lim_{s\searrow 0} s\I_s[E](q) =0.\] 
This is actually not strange at all, since in this case, actually everything could happen, depending on $\Omega, E_0$ and their respective positions. Take as an example the ``simplest'' minimal set, the half-plane. 
If $\Omega \subset \{ x_2<0\}$, then $E\cap \Omega = \Omega$, if $\Omega \subset \{ x_2>0\}$ then $E\cap \Omega= \emptyset$, while if $\Omega$ sits ``in the middle'', $E$ covers the $\Omega \cap \{ x_2<0\}$,  and it is empty in $\Omega \cap \{ x_2>0\}$.

Naturally, one may wonder what happens if \eqref{bmm} does not holds, hence if the exterior data completely surrounds $\Omega$. At least with the geometrical type of reasoning we used, in absence of \eqref{bmm} we are unable to obtain the conclusion of complete stickiness. However, only two alternatives hold: either for $s$ small enough all $s$-minimal surfaces stick or they develop a wildly oscillating behavior. Indeed, as precisely stated in \cite[Theorem 1.4 B]{bucluk}, either there exists $\sigma >0$ such that for any $s<\sigma$,  all corresponding $s$-minimal sets with exterior data $E_0$ are empty inside $\Omega$, or there exist decreasing sequences of radii $\delta_k\searrow  0$ and of parameters $s_k \searrow 0$ such that  for every corresponding $s_k$-minimal set with exterior data $E_0$, it happens that $\partial E_{s_k}$ intersects every  ball $B_{\delta_k}(x) $ compactly contained in $\Omega$. 
For further details and a thorough discussion, refer to \cite{bucluk}. 

\smallskip

To conclude this note, we reason on Example \ref{msstick1} in the nonlocal framework for $s$-small enough. The question is what happens in an unbounded domain $\Omega$ and what does complete stickiness mean in this case.

\begin{example}\label{msstick2}
Let $0<\rho<R$, $M>0$ be fixed, and let $A^R_\rho$ be the annulus
\[ A_R^\rho =\big\{ x\in \R^2  \; \big| \; \rho<|x|<R\big\}.\] 
Let 
$\varphi\colon \Rn \to \R $ be such that
	\bgs{
 		&\varphi(x)=M ,&& \mbox{ for } x\in \bar B_\rho,\\
 		 	&	\varphi(x)=0, && \mbox{ in } A^R_{R+2}  }
 		and such that at infinity, it satisfies   
 	\[	\alpha(Sg(\varphi)) <\frac{\omega_{n+1}}2,	\]
 	 for instance,  depicted in Figure \ref{stt1}.

	\begin{figure}[htpb] \label{stt1}	
    \includegraphics[width=0.8\textwidth]{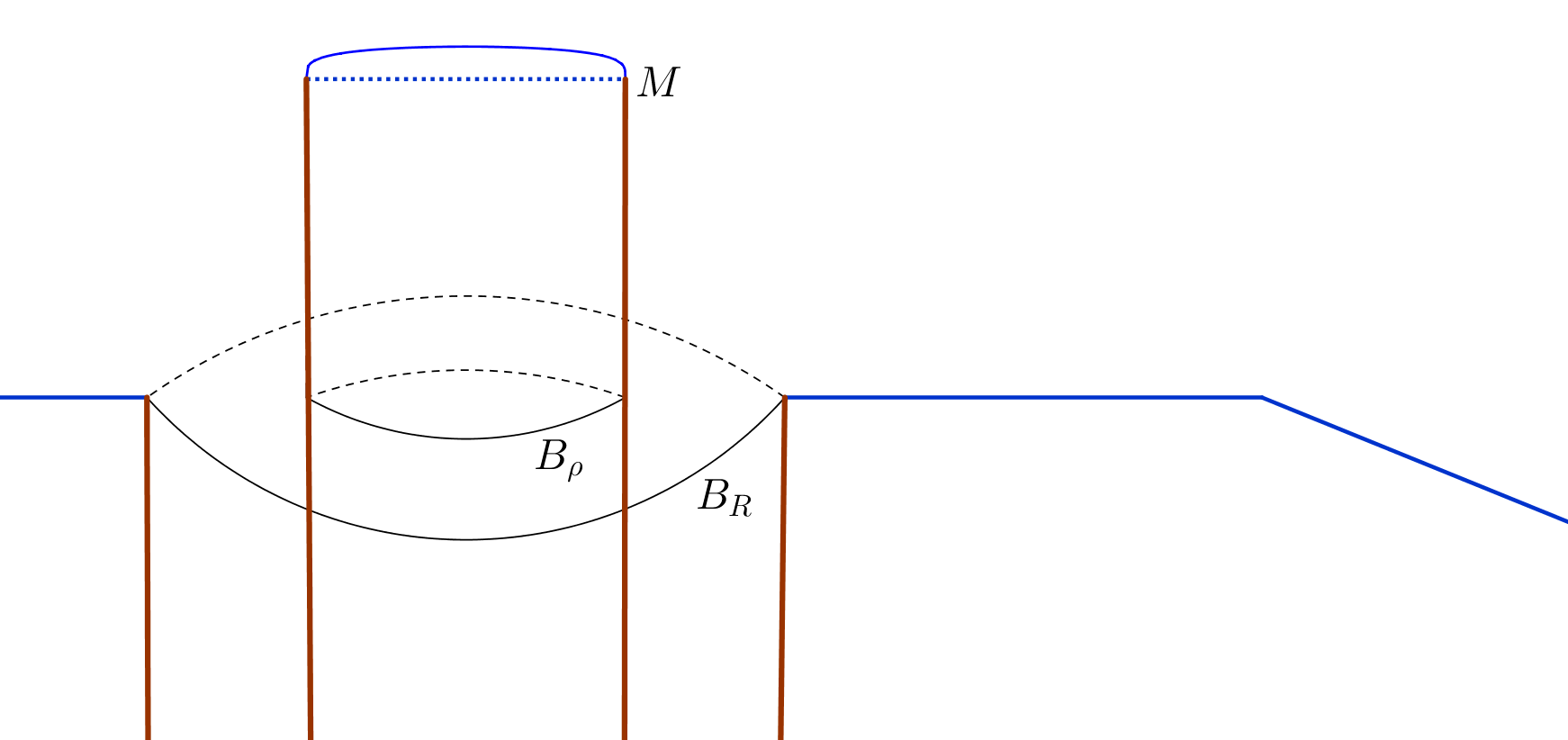} 
    \caption{Example \ref{msstick2}}   
	\label{stt1}
	
\end{figure} 
We want to minimize the $s$-perimeter in $A_R^\rho\times \R$, in the class of subgraphs with exterior data given by $\varphi$.
				What happens is that for any $K$ large enough, there exists some $s:=s(K)>0$ small enough such that 
				\[ u_s \leq -K.\]
				This means that for small values of the fractional perimeter, the stickiness occurs on both walls of the cylinder, with the height of the stickiness being as large as we want.  
				The idea of the proof starts from Theorem \ref{thm}. The exterior data does not surround the domain, thus we may start moving a ball from the outside towards the inside. There is however the challenge of the unbounded domain $ A_R^\rho\times \R$. We could solve this issue by cutting the cylinder at some height, solving the problem in the cut cylinder and then making that height as large as we want. Doing this, one should also take into account that, in principle, the data in the infinite cylinder minus the cut cylinder will contribute to $\alpha$ (this is actually negligible, since the slab has zero contribution from infinity).  However, this cutting procedure provides a non smooth  domain, thus Theorem \ref{thm} cannot be applied directly.  One could  to ``smoothen'' the domain by building ``domes'' on top of cylinders, or find a new approach to the proof that does not require a smooth domain. 
\end{example}
This discussion is developed in \cite{lukclaud},  where the authors prove a general theorem related to Example \ref{msstick2}, more precisely on the Plateau problem for nonlocal minimal graphs, with obstacles. We propose here a sketch of the theorem, referring to the original work for the complete statement, proof and further details.

\begin{theorem}
Let $\Omega\subset\R^n$ be a bounded and connected open set with $C^2$ boundary and let $\varphi:\R^n\to \R$ be such that
$$\varphi\in L^\infty_{loc}(\Rn)\quad\mbox{ and }\quad\overline{\alpha}\big(Sg(\varphi)\big)
<\frac{\omega_{n+1}}{2}.$$
Let $A\subset\subset\Omega$ be a bounded open set (eventually empty) with $C^2$ boundary. Let also
\begin{itemize}
\item [a)]  $\psi\in  C^2(\overline A)$. 
\end{itemize}
or
\begin{itemize}
\item [b)]  $\psi\in C(\overline A)\cap C^2(A)$ be such that the supgraph of $\varphi$ has $C^2$ boundary, i.e.
\[\left(\Omega\times \R \right) \setminus Sg(\varphi, \bar A)=  \big\{(x,t)\in\R^{n+1}\,\big|\,x\in\overline A,\,t> \psi(x)\big\}\]
has $C^2$ boundary. 
\end{itemize}
For every $s\in(0,1)$ we denote by $u_s$
 the unique $s$-minimal function that satisfies 
 	 \sys[]
 	 	{u_s=\varphi\quad\mbox{a.e. in }\Co\Omega\\
		u_s\ge\psi\quad\mbox{a.e. in }A.
		}
Then for every $k$ there exists $s_k\in(0,1)$ decreasing towards $0$, such that
\bgs{
u_s\le-k\quad\mbox{ a.e. in }\Omega\setminus A\quad\mbox{ and }
\quad u_s=\psi\quad\mbox{ a.e. in }A,
}
for every $s\in(0,s_k)$.
In particular
$$\lim_{s\to0}u_s(x)=-\infty,\quad\mbox{ uniformly in }x\in\Omega\setminus A.$$
\end{theorem}

In this theorem, $\varphi$ plays the role of the boundary data, whereas $\psi$ is the obstacle. 
We conclude by remarking that the $s$-minimal sets asymptotically ``empties'' the unbounded domain $\Omega$, whereas if we pick a large enough $K$, the $s$-minimal surface will stick to both walls of the cylinder,  from $-K$ until respectively reaching the boundary data $\varphi$ and the obstacle $\psi$.

\bibliography{biblio}
\bibliographystyle{plain}

\end{document}